\documentclass{amsart}

\usepackage{amsmath,amssymb}
\usepackage{amsthm}
\usepackage{amsrefs}
\usepackage{indentfirst}
\usepackage{mathrsfs}
\usepackage{hyperref}
\usepackage{xcolor}
\usepackage[margin=1.4in]{geometry}
\usepackage{nicefrac}
\usepackage[hyphenbreaks]{breakurl}

\usepackage{url}
\usepackage{algorithm}
\usepackage{algpseudocode}
\usepackage{mathtools}
\usepackage{booktabs}
\usepackage{multirow}

\numberwithin{equation}{section}
\newtheorem{theorem}{Theorem}[section]
\newtheorem{proposition}[theorem]{Proposition}
\newtheorem{lemma}[theorem]{Lemma}
\newtheorem{corollary}[theorem]{Corollary}
\newtheorem{definition}[theorem]{Definition}
\newtheorem{assumption}[theorem]{Assumption}

\allowdisplaybreaks[4]

\newcommand{\va}{\mathbf{a}}
\newcommand{\vvA}{\mathbf{A}}
\newcommand{\ve}{\mathbf{e}}

\newcommand{\vm}{\mathbf{m}}
\newcommand{\vp}{\mathbf{p}}

\newcommand{\vq}{\mathbf{q}}

\newcommand{\vu}{\mathbf{u}}

\newcommand{\vv}{\mathbf{v}}
\newcommand{\vw}{\mathbf{w}}
\newcommand{\vx}{\mathbf{x}}

\newcommand{\vy}{\mathbf{y}}
\newcommand{\bvy}{\boldsymbol{y}}
\newcommand{\vz}{\mathbf{z}}

\newcommand{\vzero}{\mathbf{0}}

\newcommand{\vA}{\mathbf{A}}

\newcommand{\vI}{\mathbf{I}}

\newcommand{\vU}{\mathbf{U}}

\newcommand{\vW}{\mathbf{W}}

\newcommand{\vX}{\mathbf{X}}
\newcommand{\vY}{\mathbf{Y}}
\newcommand{\vepsilon}{\boldsymbol{\epsilon}}

\newcommand{\calA}{\mathcal{A}}
\newcommand{\calD}{\mathcal{D}}
\newcommand{\calI}{\mathcal{I}}
\newcommand{\calM}{\mathcal{M}}
\newcommand{\calN}{\mathcal{N}}
\newcommand{\calO}{\mathcal{O}}
\newcommand{\calP}{\mathcal{P}}
\newcommand{\calR}{\mathcal{R}}
\newcommand{\calT}{\mathcal{T}}

\newcommand{\tto}{{(t)}}
\newcommand{\bR}{\mathbb{R}}
\newcommand{\bE}{\mathbb{E}}

\newcommand{\avetT}{\frac{1}{T}\sum_{t=1}^T}
\newcommand{\sumjk}{\sum_{j=1}^k}

\newcommand{\ie}{{\it i.e.}}

\newcommand{\argmin}{\mathop{\rm argmin}}
\DeclareMathOperator*{\Min}{minimize}

\title[Efficient Algorithms for Sum-of-Minimum Optimization]{Efficient Algorithms for Sum-of-Minimum Optimization}
\author{Lisang Ding}
\address{(LD) Department of Mathematics, University of California, Los Angeles (UCLA), Los Angeles, CA 90095.}
\email{lsding@math.ucla.edu}
\author{Ziang Chen}
\address{(ZC) Department of Mathematics, Massachusetts Institute of Technology (MIT), Cambridge, MA 02139.}
\email{ziang@mit.edu}
\author{Xinshang Wang}
\address{(XW) Decision Intelligence Lab, Damo Academy, Alibaba US, Bellevue, WA 98004.}
\email{xinshang.w@alibaba-inc.com}
\author{Wotao Yin}
\address{(WY) Decision Intelligence Lab, Damo Academy, Alibaba US, Bellevue, WA 98004.}
\email{wotao.yin@alibaba-inc.com}
\date{\today}

\thanks{A major part of the work of ZC was completed during his internship at Alibaba US DAMO Academy. Corresponding author: Ziang Chen, ziang@mit.com}

\begin{document}
	\begin{abstract}
		In this work, we propose a novel optimization model termed ``sum-of-minimum" optimization. This model seeks to minimize the sum or average of $N$ objective functions over $k$ parameters, where each objective takes the minimum value of a predefined sub-function with respect to the $k$ parameters. This universal framework encompasses numerous clustering applications in machine learning and related fields. We develop efficient algorithms for solving sum-of-minimum optimization problems, inspired by a randomized initialization algorithm for the classic $k$-means \cite{arthur2007k} and Lloyd's algorithm \cite{lloyd1982least}. We establish a new tight bound for the generalized initialization algorithm and prove a gradient-descent-like convergence rate for generalized Lloyd's algorithm. The efficiency of our algorithms is numerically examined on multiple tasks, including generalized principal component analysis, mixed linear regression, and small-scale neural network training. Our approach compares favorably to previous ones based on simpler-but-less-precise optimization reformulations.
	\end{abstract}
	
	\maketitle
	
	\section{Introduction}
	In this paper, we propose the following ``sum-of-minimum" optimization model:
	\begin{equation}
		\label{eq:sum-of-min-opt}
		\Min_{\vx_1,\vx_2,\ldots,\vx_k}~  F(\vx_1,\vx_2,\ldots,\vx_k):=  \frac{1}{N}\sum_{i=1}^N\min\{ f_i(\vx_1),f_i(\vx_2),\ldots,f_i(\vx_k) \},
	\end{equation}
	where $\vx_1,\vx_2,\ldots,\vx_k$ are unknown parameters to determine. The cost function $F$ is the average of $N$ objectives where the $i$-th objective is $f_i$ evaluated at its ``optimal" out of the $k$ parameter choices.  
	This paper aims to develop efficient algorithms for solving \eqref{eq:sum-of-min-opt} and analyze their performance. 
	
	Write $[k]=\{1,2,\dots,k\}$ and $[N]=\{1,2,\dots,N\}$. Let $(C_1,C_2,\ldots,C_k)$ be a partition of $[N]$, \ie, $C_i$'s are disjoint subsets of $[N]$ and their union equals $[N]$. Let $\calP^k_N$ denote the set of all such partitions. 
	Then, \eqref{eq:sum-of-min-opt} is equivalent to
	\begin{equation}
		\label{eq:sum-of-min-opt-2}
		\Min_{(C_1,C_2,\ldots,C_k)\in\calP^k_N}\min_{\vx_1,\vx_2,\ldots,\vx_k}\frac{1}{N} \sum_{j=1}^k\sum_{i\in C_j} f_i(\vx_j).
	\end{equation}
	It is easy to see $(C_1^*,C_2^*,\ldots,C_k^*)$ and $(\vx_1^*,\vx_2^*,\ldots,\vx_k^*)$ are optimal to \eqref{eq:sum-of-min-opt-2} if and only if $(\vx_1^*,\vx_2^*,\ldots,\vx_k^*)$ is optimal to \eqref{eq:sum-of-min-opt} and
	\begin{equation*}
		i\in C_j^* \implies f_i(\vx_j^*) = \min\{ f_i(\vx_1^*),f_i(\vx_2^*),\ldots,f_i(\vx_k^*) \}.
	\end{equation*}
	Reformulation \eqref{eq:sum-of-min-opt-2} reveals its clustering purpose. It finds the optimal partition $(C_1^*,C_2^*,\ldots,C_k^*)$ such that using the parameter $\vx_j^*$ to minimize the average of $f_i$'s in the cluster $C_j$ leads to the minimal total cost.
	
	Problem \eqref{eq:sum-of-min-opt} generalizes \textbf{$k$-means clustering}. Consider $N$ data points $\vy_1,\vy_2,\ldots,\vy_N$ and a distance function $d(\cdot,\cdot)$. The goal of $k$-means clustering is to find clustering centroids $\vx_1,\vx_2,\ldots,\vx_k$ that minimize
	$$
	F(\vx_1,\vx_2,\ldots,\vx_k)=\frac{1}{N}\sum_{i=1}^N \min_{j\in[k]}\{d(\vx_j,\vy_i)\},
	$$
	which is the average distance from each data point to its nearest cluster center. The literature presents various choices for the distance function $d(\cdot,\cdot)$.
	When $d(\vx,\vy)=\frac{1}{2}\|\vx-\vy\|^2$, this optimization problem reduces to the classic $k$-means clustering problem, for which numerous algorithms have been proposed \cites{krishna1999genetic,arthur2007k,na2010research,sinaga2020unsupervised,ahmed2020k}. Bregman divergence is also widely adopted as a distance measure\cites{banerjee2005clustering,manthey2013worst,liu2016clustering}, defined as 
	$$
	d(\vx,\vy)=h(\vx)-h(\vy)-\langle \nabla h(\vy),\vx-\vy \rangle,
	$$
	with $h$ being a differentiable convex function. 
	
	A special case of \eqref{eq:sum-of-min-opt} is \textbf{mixed linear regression}, which generalizes linear regression and models the dataset $\{(\va_i,b_i)\}_{i=1}^N$ by multiple linear models. A linear model is a function $g(\va;\vx)=\va^\top \vx$, which utilizes $\vx$ as the coefficient vector for each model. Make $k$ copies of the linear model and set the $j$-th linear coefficient as $\vx_j$. The loss for each data pair $(\va_i,b_i)$ is computed as the squared error from the best-fitting linear model, specifically $\min_{j\in[k]}\{\frac{1}{2}(g(\va_i;\vx_j)-b_i)^2\}$. We aim to search for optimal parameters $\{\vx_j\}_{j=1}^k$ that minimizes the average loss
	\begin{equation}
		\label{eq:mixed-linear-regression}
		\frac{1}{N}\sum_{i=1}^N \min_{j\in[k]}\left\{  \frac{1}{2}\left(g(\va_i;\vx_j)-b_i\right)^2 \right\}.
	\end{equation}
	Paper \cite{zhong2016mixed} simplifies this non-smooth problem to the sum-of-product problem:
	\begin{equation}
		\label{eq:mixed-linear-regression-2}
		\Min_{\vx_1,\vx_2,\dots,\vx_k}~\frac{1}{N}\sum_{i=1}^N \prod_{j\in[k]}\left(g(\va_i;\vx_j)-b_i\right)^2,
	\end{equation}
	which is smooth. Although \eqref{eq:mixed-linear-regression-2} is easier to approach due to its smooth objective function, problem \eqref{eq:mixed-linear-regression} is more accurate. Various algorithms are proposed to recover $k$ linear models from mixed-class data~\cites{yi2014alternating,shen2019iterative,kong2020meta,zilber2023imbalanced}.
	
	In \eqref{eq:mixed-linear-regression}, the function $g(\cdot;\vx)$ parameterized by $\vx$ can be any nonlinear function such as neural networks, and we call this extension \textbf{mixed nonlinear regression}.
	
	An application of \eqref{eq:sum-of-min-opt} is \textbf{generalized principal component analysis (GPCA)}~\cites{vidal2005generalized,tsakiris2017filtrated}, which aims to recover $k$ low-dimensional subspaces, $V_1, V_2, \dots, V_k$, from the given data points $\vy_1,\vy_2,\dots,\vy_N$, which are assumed to be located on or close to the collective union of these subspaces $V_1\cup V_2\cup\cdots\cup V_k$. This process, also referred to as subspace clustering, seeks to accurately segment data points into their respective subspaces~\cites{ma2008estimation,vidal2011subspace,elhamifar2013sparse}. Each subspace $V_j$ is represented as $V_j = \{\vy\in\bR^d: \vy^\top \vvA_j = 0\}$ where $\vvA_j\in \bR^{d\times r}$ and $\vvA_j^\top \vvA_j = I_r$, with $r$ being the co-dimension of $V_j$. From an optimization perspective, the GPCA task can be formulated as
	\begin{equation}
		\label{eq:GPCA}
		\Min_{\vvA_j^\top \vvA_j = I_r}~\frac{1}{N}\sum_{i=1}^N \min_{j\in[k]}\left\{\frac{1}{2}\|\vy_i^\top \vvA_j\|^2\right\}.
	\end{equation}
	Similar to \eqref{eq:mixed-linear-regression-2}, \cite{peng2023block} works with the less precise reformulation using the product of $\|\vy_i^\top \vvA_j\|^2$ for smoothness and introduces block coordinate descent algorithm.

	When $k=1$, problem \eqref{eq:sum-of-min-opt} reduces to the finite-sum optimization problem
	\begin{equation}\label{eq:sum-opt}
		\min_{\vx}~F(\vx) = \frac{1}{N} \sum_{i=1}^N f_i(\vx),
	\end{equation}
	widely used to train machine learning models,
	where $f_i(\vx)$ depicts the loss of the model at parameter $\vx$ on the $i$-th data point.
	When the underlying model lacks sufficient expressiveness, problem ~\eqref{eq:sum-opt} alone may not yield satisfactory results.To enhance a model's performance, one can train the model with multiple parameters, $\vx_1,\vx_2,\cdots,\vx_k, k\geq 2$, and utilize only the most effective parameter for every data point. This strategy has been successfully applied in various classic tasks, including the aforementioned $k$-means clustering,
	mixed linear regression,
	and the generalized principal component analysis.  
	These applications share a common objective: to segment the dataset into $k$ groups and identify the best parameter for each group.
	Although no single parameter might perform well across the entire dataset, every data point is adequately served by at least one of the $k$ parameters.
	By aggregating the strengths of multiple smaller models, this approach not only enhances model expressiveness but also offers a cost-efficient alternative to deploying a singular larger model.

	Although one might expect that algorithms and analyses for the sum-of-minimum problem \eqref{eq:sum-of-min-opt} to be weaker as \eqref{eq:sum-of-min-opt} subsumes the discussed previous models, we find our algorithms and analyses for \eqref{eq:sum-of-min-opt} to enhance those known for the existing models. Our algorithms extend the \texttt{$k$-means++} algorithm~\cite{arthur2007k} and Lloyd's algorithm~\cite{lloyd1982least}, which are proposed for classic $k$-means problems. We obtain new bounds of these algorithms for \eqref{eq:sum-of-min-opt}. Our contributions are summarized as follows:
	\begin{itemize}
		\item We propose the sum-of-minimum optimization problem, adapt \texttt{$k$-means++} to the problem for initialization, and generalize Lloyd's algorithm to approximately solve the problem. 
		\item We establish theoretical guarantees for the proposed algorithms. Specifically, under the assumption that each $f_i$ is $L$-smooth and $\mu$-strongly convex, we prove the output of the initialization is $\calO(\frac{L^2}{\mu^2}\ln k)$-optimal and that this bound is tight with respect to both $k$ and the condition number $\frac{L}{\mu}$. When reducing to $k$-means optimization, our result recovers that of~\cite{arthur2007k}. Furthermore, we prove an $\calO(\frac{1}{T})$ convergence rate for generalized Lloyd's algorithms.
		\item We numerically verify the efficiency of the proposed framework and algorithms on several tasks, including generalized principal component analysis, $\ell_2$-regularized mixed linear regression, and small-scale neural network training. The results reveal that our optimization model and algorithm lead to a higher successful rate in finding the ground-truth clustering, compared to existing approaches that resort to less accurate reformulations for the sake of smoother optimization landscapes. Moreover, our initialization shows significant improvements in both convergence speed and chance of obtaining better minima.
	\end{itemize}
	Our work significantly generalizes classic $k$-means to handles more complex nonlinear models and provides new perspectives for improving the model performance.
	The rest of this paper is organized as follows. We introduce the preliminaries and related works in Section~\ref{sec:preliminary}. We present the algorithms in Section~\ref{sec:alg}. The algorithms are analyzed theoretically in Section~\ref{sec:theory} and numerically in Section~\ref{sec:num-exp}. The paper is concluded in Section~\ref{sec:conclude}.
	
	Throughout this paper, the $\ell_2$-norm and $\ell_2$-inner product are denoted by $\|\cdot\|$ and $\langle\cdot,\cdot\rangle$, respectively. We employ $|\cdot|$ as the cardinal number of a set.
	
	\section{Related Work and Preliminary}
	\label{sec:preliminary}
	\subsection{Related work}
	Lloyd's algorithm~\cite{lloyd1982least}, a well-established iterative method for the classic $k$-means problem, alternates between two key steps~\cite{mackay2003example}:
	1) assigning $\vy_i$ to $\vx_j^{(t)}$ if $\vx_j^{(t)}$ is the closest to $\vy_i$ among $\{\vx_1^{(t)},\vx_2^{(t)},\dots,\vx_k^{(t)}\}$; 2) updating $\vx_j^{(t+1)}$ as the centroid of all $\vy_i$'s assigned to $\vx_j^{(t)}$. Although Lloyd's algorithm can be proved to converge to stationary points, the results can be highly suboptimal due to the inherent non-convex nature of the problem. Therefore, the performance of Lloyd's algorithm highly depends on the initialization. To address this, a randomized initialization algorithm, 
	\texttt{$k$-means++} \cite{arthur2007k}, generates an initial solution in a sequential fashion. Each centroid $\vx_j^{(0)}$ is sampled recurrently according to the distribution
	\begin{equation}\label{eq:kmeans++}
		\mathbb{P}(\vx_j^{(0)} = \vy_i) \propto \min_{1\leq j'\leq j-1} \|\vx_{j'}-\vy_i\|^2,\quad i\in[N].
	\end{equation}
	The idea is to sample a data point farther from the current centroids with higher probability, ensuring the samples to be more evenly distributed across the dataset. It is proved in \cite{arthur2007k} that
	\begin{equation}
		\label{eq:k-means++bound}
		\mathbb{E} F(\vx_1^{(0)},\vx_2^{(0)},\ldots,\vx_k^{(0)}) \leq 8(\ln k+2) F^*,
	\end{equation}
	where $F^*$ is the optimal objective value of $F$. This seminal work has inspired numerous enhancements to the \texttt{$k$-means++} algorithm, as evidenced by contributions from \cites{bahmani2012scalable,zimichev2014spectral,bachem2016fast,bachem2016approximate,wu2021user,ren2022novel}.
	Our result generalizes the bound in \eqref{eq:k-means++bound}, broadening its applicability in sum-of-minimum optimization.

	\subsection{Definitions and assumptions}
	
	In this subsection, we outline the foundational settings for our algorithm and theory. For each sub-function $f_i$, we establish the following assumptions.
	
	\begin{assumption}
		\label{asp:L-smooth}
		Each $f_i$ is $L$-smooth, satisfying
		\begin{equation*}
			\|\nabla f_i(x) - \nabla f_i(y)\|\leq L\|x-y\|,\quad\forall~x,y\in\bR^d,\ i\in[N]. 
		\end{equation*}
		
	\end{assumption}
	\begin{assumption}
		\label{asp:mu-strong-convex}
		Each $f_i$ is $\mu$-strongly convex, for all $x,y\in\bR^d$ and $i\in[N]$,
		\begin{equation*}
			f_i(y)\geq f_i(x) + \nabla f_i(x)^\top (y-x) + \frac{\mu}{2} \|x-y\|^2.
		\end{equation*}
	\end{assumption}
	
	Let $\vx^*_i$ denote the optimizer of $f_i(\vx)$ such that $f_i^*=f_i(\vx^*_i)$, and let
	\begin{equation*}
		S^*=\{\vx_i^*:1\leq i\leq N\}
	\end{equation*}
	represent the solution set. If $S^*$ comprises $l<k$ different elements, the problem \eqref{eq:sum-of-min-opt} possesses infinitely many global minima. Specifically, we can set the variables $\vx_1,\vx_2,\ldots,\vx_l$ to be the $l$ distinct elements in $S^*$, while leaving $\vx_{l+1},\vx_{l+2},\ldots,\vx_k$ as free variables. Given these $k$ variables, $F(\vx_1,\vx_2,\ldots,\vx_k)=\frac{1}{N}\sum_{i=1}^N f_i^*$. If $S^*$ contains more than $k$ distinct components, we have the following proposition.
	
	\begin{proposition}
		\label{prop:S-k-sep}
		Under Assumption \ref{asp:mu-strong-convex}, if $|S^*|\geq k$, the optimization problem \eqref{eq:sum-of-min-opt} admits finitely many minimizers.
	\end{proposition}
	
	Expanding on the correlation between the number of global minimizers and the size of $S^*$, we introduce well-posedness conditions for $S^*$.
	\begin{definition}[$k$-separate and $(k,r)$-separate]
		We call $S^*$ $k$-separate if it contains at least $k$ different elements, \ie, $|S^*|\geq k$.
		Furthermore, we call $S^*$ $(k,r)$-separate if there exists $1\leq i_1<i_2<\cdots<i_k\leq N$ such that $\|\vx_{i_j}^*-\vx_{i_{j'}}^*\|> 2r$ for all $j\not=j'$.
	\end{definition}
	
	Finally, we address the optimality measurement in \eqref{eq:sum-of-min-opt}. The norm of the (sub)-gradient is an inappropriate measure for global optimality due to the problem's non-convex nature. Instead, we utilize the following optimality gap.
	
	\begin{definition}[Optimality gap]
		Given a point $\vx$, the optimality gap of $f_i$ at $\vx$ is $f_i(\vx)-f^*_i$.
		Given a finite point set $\calM$, the optimality gap of $f_i$ at $\calM$ is
		$\min_{\vx\in\calM} f_i(\vx)-f^*_i$.
		When $\calM=\{\vx_1,\vx_2,\dots,\vx_k\}$, the averaged optimality gap of $f_1,f_2,\dots,f_N$ at $\calM$ is the shifted objective function
		\begin{equation}\label{eq:average-opt-gap}
			F(\vx_1,\vx_2,\dots,\vx_k)-\frac{1}{N}\sum_{i=1}^N f_i^*.
		\end{equation}
	\end{definition}
	
	The averaged optimality gap in \eqref{eq:average-opt-gap} will be used as the optimality measurement throughout this paper. Specifically, in the classic $k$-means problem, one has $f_i^*=0$, so the function $F(\vx_1,\vx_2,\dots,\vx_k)$ directly indicates global optimality.
	
	\section{Algorithms}
	\label{sec:alg}
	In this section, we introduce the algorithm for solving the sum-of-minimum optimization problem \eqref{eq:sum-of-min-opt}. Our approach is twofold, comprising an initialization phase based on \texttt{$k$-means++} and a generalized version of Lloyd's algorithm.
	
	\subsection{Initialization} As the sum-of-minimum optimization \eqref{eq:sum-of-min-opt} can be considered a generalization of the classic $k$-means clustering, we adopt \texttt{$k$-means++}. In \texttt{$k$-means++}, clustering centers are selected sequentially from the dataset, with each data point chosen based on a probability proportional to its squared distance from the nearest existing clustering centers, as detailed in \eqref{eq:kmeans++}. We generalize this idea and propose the following initialization algorithm that outputs initial parameters $\vx_1^{(0)},\vx_2^{(0)},\dots,\vx_k^{(0)}$ for the problem \eqref{eq:sum-of-min-opt}.
	
	First, we select an index $i_1$ at random from $[N]$, following a uniform distribution, and then utilize a specific method to determine the minimizer $\vx^*_{i_1}$, setting
	\begin{equation}\label{eq:x10}
		\vx_1^{(0)}=\vx^*_{i_1} = \argmin_\vx f_{i_1}(\vx).
	\end{equation}
	For $j=2,3,\dots,k$, we sample $i_j$ based on the existing variables $\calM_j=\{\vx_{1}^{(0)},\vx_{2}^{(0)},\ldots,\vx_{j-1}^{(0)}\}$, with each index $i$ sampled based on a probability proportional to the optimality gap of $f_i$ at $\calM_j$.
	Specifically, we compute the minimal optimality gaps
	\begin{equation}
		\label{eq:v-score-A}
		v_i^{(j)}=\min_{1\leq j'\leq j-1}\left(f_i(\vx^{(0)}_{j'})-f^*_i \right),\quad i\in[N],
	\end{equation}
	as probability scores. Each score $v_i^{(j)}$ can be regarded as an indicator of how unresolved an instance $f_i$ is with the current variables $\{\vx_{j'}^{(0)} \}_{j'=1}^{j-1}$. We then normalize these scores
	\begin{equation}\label{eq:w-wight}
		w_i^{(j)}=\frac{v_i^{(j)}}{\sum_{i'=1}^N v_{i'}^{(j)}},\quad i\in[N],
	\end{equation}
	and sample $i_j\in[N]$ following the probability distribution $\vw^{(j)}=\left(w^{(j)}_1, \dots,w^{(j)}_N\right)$. The $j$-th initialization is determined by optimizing $f_{i_j}$,
	\begin{equation}\label{eq:xj0}
		\vx_j^{(0)}= \vx_{i_j}^* = \argmin_\vx f_{i_j}(\vx).
	\end{equation}
	We terminate the selection process once $k$ variables $\vx_1^{(0)},\vx_2^{(0)},\ldots,\vx_k^{(0)}$ are determined. The pseudo-code of this algorithm is shown in Algorithm \ref{alg:init-A}. 
	
	\begin{algorithm}[htb!]
		\caption{Initialization}
		\label{alg:init-A}
		\begin{algorithmic}[1]
			\State Sample $i_1$ uniformly at random from $[N]$ and compute $\vx_1^{(0)}$ via \eqref{eq:x10}.
			\For{$j=2,3,\dots,k$}
			\State Compute $\vv^{(j)}=\left(v_1^{(j)},v_2^{(j)},\ldots,v_N^{(j)}\right)$ via \eqref{eq:v-score-A}.
			\State Compute $\vw^{(j)}=\left(w^{(j)}_1, \dots,w^{(j)}_N\right)$ via \eqref{eq:w-wight}.
			\State Sample $i_j\in[N]$ according to the weights $\vw^{(j)}$ and compute $\vx_j^{(0)}$ via \eqref{eq:xj0}.
			\EndFor
		\end{algorithmic}
	\end{algorithm}
	
	We note that the scores $v_i^{(j)}$ defined in \eqref{eq:v-score-A} rely on the optimal objectives $f_i^*$, which may be computationally intensive to calculate in certain scenarios. Therefore, we propose a variant of Algorithm~\ref{alg:init-A} by adjusting the scores $v_i^{(j)}$. Specifically, when $j-1$ parameters $\vx_1^{(0)},\vx_2^{(0)},\ldots,\vx_{j-1}^{(0)}$ are selected, the score is set as the minimum squared norm of the gradient:
	\begin{equation}
		\label{eq:v-score-D}
		v_i^{(j)}=\min_{1\leq j' \leq j-1} \left\| \nabla f_i(\vx_{j'}^{(0)}) \right\|^2.
	\end{equation}
	This variant involves replacing the scores in Step 3 of Algorithm~\ref{alg:init-A} with \eqref{eq:v-score-D}, which is further elaborated in Appendix \ref{sec:alg-detail}. 
	
	In the context of classic $k$-means clustering where $f_i(\vx)=\frac{1}{2}\|\vx-\vy_i\|^2$ for the $i$-th data point $\vy_i$, the score $v_i^{(j)}$ in both \eqref{eq:v-score-A} and \eqref{eq:v-score-D} reduces to
	$$
	\min_{1\leq j'\leq j-1}\|\vx_{j'}^{(0)}-\vy_i\|^2,
	$$
	up to a constant scalar. This initialization algorithm, whether utilizing scores from \eqref{eq:v-score-A} or \eqref{eq:v-score-D}, aligns with the approach of the classic \texttt{$k$-means++} algorithm.

	\subsection{Generalized Lloyd's algorithm}
	Lloyd's algorithm is employed to minimize the loss in $k$-means clustering by alternately updating the clusters and their centroids~\cites{lloyd1982least,mackay2003example}. This centroid update process can be regarded as a form of gradient descent applied to group functions, defined by the average distance between data points within a cluster and its centroid~\cite{bottou1994convergence}.  
	For our problem \eqref{eq:sum-of-min-opt}, we introduce a novel gradient descent algorithm that utilizes dynamic group functions. Our algorithm is structured into two main phases: reclassification and group gradient descent.
	
	\paragraph{Reclassification.} The goal is for $C_j^{(t)}$ to encompass all $i\in[N]$ where $f_i$ is active at $\vx_j^{(t)}$, allowing us to use the sub-functions $f_i$ within $C_j^{(t)}$ to update $\vx_j^{(t)}$. This process leads to the reclassification step as follows:
	\begin{equation}
		\label{eq:C-partition-update}
		C_j^{(t)} = \Big\{i\in[N] : f_i(\vx_j^{(t)})\leq f_i(\vx_{j'}^{(t)}),
		\forall~ j'\in[k]\Big\} \Big\backslash\Big(\bigcup_{l< j}C_l^{(t)}\Big),\quad j=1,2,\dots,k.
	\end{equation}
	Given that reclassification may incur non-negligible costs in practice, a reclassification frequency $r$ can be established, performing the update in \eqref{eq:C-partition-update} every $r$ iterations while keeping $C_j^{(t)}=C_j^{(t-1)}$ constant during other iterations. 
	
	\paragraph{Group gradient descent.} With $C_j^{(t)}$ indicating the active $f_i$ at $\vx_j^{(t)}$, 
	we can define the group objective function:
	\begin{equation}
		\label{eq:F-j-t}
		F_j^{(t)}(\vz)=
		\begin{cases}
			\frac{1}{|C_j^{(t)}|}\sum_{i \in C_j^{(t)}} f_i(\vz), & C_j^{(t)}\not=\emptyset,\\
			0, & C_j^{(t)}=\emptyset,
		\end{cases}
	\end{equation}
	In each iteration, gradient descent is performed on $\vx_j^\tto$ individually as:
	\begin{equation}
		\label{eq:group_gd}
		\vx_j^{(t+1)}=\vx_j^{(t)}-\gamma \nabla F_j^{(t)}(\vx_j^{(t)}).
	\end{equation}
	Here, $\gamma>0$ is the chosen step size.
	Alternatively, one might opt for different iterative updates or directly compute:
	\begin{equation*}
		\vx_j^{(t+1)} = \argmin_\vx \sum_{i \in C_j^{(t)}} f_i(\vx),
	\end{equation*}
	especially if the minimizer of $\sum_{i \in C_j^{(t)}} f_i(\vx)$ admits a closed form or can be computed efficiently. The pseudo-code consisting of the above two steps is presented in Algorithm \ref{alg:Lloyd}.

	\begin{algorithm}[htb!]
		\caption{Generalized Lloyd's Algorithm}
		\label{alg:Lloyd}
		\begin{algorithmic}[1]
			\State Generate the initialization $\vx_1^{(0)},\vx_2^{(0)},\dots,\vx_k^{(0)}$ and set $r,\gamma$.
			\For{$t=0,1,2,\ldots,T$}
			\If{$t\equiv 0\, (\textup{mod }r)$}
			\State Compute the partition $\{C_j^{(t)}\}_{j=1}^k$ via \eqref{eq:C-partition-update}.
			\Else
			\State{$C_j^{(t)}=C_j^{(t-1)},\quad 1\leq j\leq k.$}
			\EndIf
			\State Compute $\vx_j^{(t+1)}$ via \eqref{eq:group_gd}.
			\EndFor
		\end{algorithmic}
	\end{algorithm}

	\textbf{Momentum Lloyd's Algorithm.} We enhance Algorithm \ref{alg:init-A} by incorporating a momentum term. The momentum for $\vx_j^\tto$ is represented as $\vm_j^\tto$, with $0<\beta<1$ and $\gamma>0$ serving as the step sizes for the momentum-based updates. We use the gradient of the group function $F_j^\tto$ to update the momentum $\vm_j^\tto$. The momentum algorithm admits the following form:
	\begin{align}
		\label{eq:momentum-x}
		&\vx_j^{(t+1)}=\vx_j^\tto-\gamma \vm_j^\tto,\\
		\label{eq:momentum-m}
		&\vm_j^{(t+1)}=\beta \vm_j^\tto+\nabla F_j^{(t+1)}(\vx_j^{(t+1)}).
	\end{align}
	A critical aspect of the momentum algorithm involves updating the classes $C_j^\tto$ between \eqref{eq:momentum-x} and \eqref{eq:momentum-m}. Rather than reclassifying based on $f_i$ evaluated at $\vx_j^{(t+1)}$, reclassification leverages an acceleration variable:
	\begin{equation}
		\label{eq:u-acceleration-var}
		\vu_j^{(t+1)}=\frac{1}{1-\beta}(\vx_j^{(t+1)}- \beta \vx_j^\tto).
	\end{equation}
	The index $i$ will be classified to $C_j^{(t+1)}$ where $f_i(\vu_j^{(t+1)})$ attains the minimal value.
	Furthermore, to mitigate abrupt shifts in each class $C_j$, we implement a \textit{controlled} reclassification scheme that limits the extent of change in each class:
	\begin{equation}
		\label{eq:control-C}
		\frac{1}{\alpha}|C_j^\tto|\leq |C_j^{(t+1)}|\leq \alpha |C_j^\tto|,
	\end{equation}
	where $\alpha>1$ serves as a constraint factor. Details of the momentum algorithm are provided in Appendix \ref{sec:alg-detail}. We display the pseudo-code in Algorithm \ref{alg:Lloyd-momentum}.
	
	\begin{algorithm}[t!]
		\caption{Momentum Lloyd's Algorithm}\label{alg:Lloyd-momentum}
		\begin{algorithmic}[1]
			\State Generate the initialization $\vx_1^{(0)},\vx_2^{(0)},\dots,\vx_k^{(0)}$. Set $\vm_1^{(0)},\vm_2^{(0)},\ldots,\vm_k^{(0)}$ to be $\vzero$. Set $r,\alpha,\beta,\gamma$.
			\For{$t=0,1,2,\ldots,T$}
			\State{Update $\vx_j^\tto$ using \eqref{eq:momentum-x}.}
			\If{$t\equiv 0\, (\textup{mod }r)$}
			\State Compute $\vu_j^{(t+1)}$ via \eqref{eq:u-acceleration-var}.
			\State Update $C_j^{(t+1)}$ with $\vu_j^{(t+1)}$ in control, such that \eqref{eq:control-C} holds.
			\Else
			\State{$C_j^{(t+1)}=C_j^{(t)},\quad 1\leq j\leq k.$}
			\EndIf
			\State Update the momentum $\vm_j^\tto$ via \eqref{eq:momentum-m}.
			\EndFor
		\end{algorithmic}
	\end{algorithm} 
	
	\section{Theoretical Analysis}
	\label{sec:theory}
	In this section, we prove the efficiency of the initialization algorithm and establish the convergence rate of Lloyd's algorithm. For the initialization Algorithm \ref{alg:init-A}, we show that the ratio between the optimality gap of $\{\vx_1^{(0)},\vx_2^{(0)},\dots,\vx_k^{(0)}\}$ and the smallest possible optimality gap is $\calO(\frac{L^2}{\mu^2}\ln k)$. Additionally, by presenting an example where this ratio is $\Omega(\frac{L^2}{\mu^2}\ln k)$, we illustrate the bound's tightness. For Lloyd's Algorithms \ref{alg:Lloyd} and \ref{alg:Lloyd-momentum}, we establish a gradient decay rate of $\calO(\frac{1}{T})$, underscoring the efficiency and convergence properties of these algorithms.

	\subsection{Error bound of the initialization algorithm}
	We define the set of initial points selected by the randomized initialization Algorithm \ref{alg:init-A},
	$$
	\calM_{\textup{init}}=\{\vx_1^{(0)},\vx_2^{(0)},\dots,\vx_k^{(0)}\}=\{ \vx_{i_1}^*,\vx_{i_2}^*,\ldots, \vx_{i_k}^* \},
	$$
	as the starting configuration for our optimization process. For simplicity, we use $F(\calM_{\textup{init}})=F(\vx_{i_1}^*,\vx_{i_2}^*,\ldots, \vx_{i_k}^*)$ to represent the function value at these initial points. Let $F^*$ be the global minimal value of $F$, and let $f^*=\frac{1}{N}\sum_{i=1}^Nf_i^*$ denote the average of the optimal values of sub-functions. The effectiveness of Algorithm~\ref{alg:init-A} is evaluated by the ratio between $\bE F(\calM_{\textup{init}}) - f^*$ and $F^* - f^*$,
	which is the expected ratio between the averaged optimality gap at $\calM_{\textup{init}}$ and the minimal possible averaged optimality gap. The following theorem provides a specific bound.
	
	\begin{theorem}
		\label{thm:init-upper-bound}
		Suppose that Assumptions \ref{asp:L-smooth} and \ref{asp:mu-strong-convex} hold. Assume that the solution set $S^*$ is $k$-separate. Let $\calM_{\textup{init}}$ be a random initialization set generated by Algorithm \ref{alg:init-A}. We have
		\begin{equation*}
			\bE F(\calM_{\textup{init}}) - f^*
			\leq
			4(2+\ln k) \left( \frac{L^2}{\mu^2}+\frac{L}{\mu} \right)\left(F^* - f^*\right).
		\end{equation*}
	\end{theorem}
	
	Theorem~\ref{thm:init-upper-bound} indicates that the relative optimality gap at the initialization set is constrained by a factor of $\calO( \frac{L^2}{\mu^2}\ln k)$ times the minimal optimality gap. The proof of Theorem~\ref{thm:init-upper-bound} is detailed in Appendix~\ref{sec:init-error-bound}. In the classic $k$-means problem, where $L=\mu$, this result reduces to Theorem 1.1 in \cite{arthur2007k}. Moreover, the upper bound $\calO( \frac{L^2}{\mu^2}\ln k)$ is proven to be tight via a lower bound established in the following theorem.
	
	\begin{theorem}
		\label{thm:init-lower-bound}
		Given a fixed cluster number $k>0$, there exists an integer $N>0$. We can construct $N$ sub-functions $\{f_i\}_{i=1}^N$ satisfying Assumptions \ref{asp:L-smooth}--\ref{asp:mu-strong-convex} and guaranteeing the solution set $S^*$ to be $k$-separate. 
		When applying Algorithm \ref{alg:init-A} over the instances $\{f_i\}_{i=1}^N$, we have
		\begin{equation}
			\label{eq:A-init-lb}
			\bE F(\calM_{\textup{init}}) - f^* \geq \frac{1}{2}\frac{L^2}{\mu^2}\ln k \left(F^* - f^*\right).
		\end{equation}
	\end{theorem}

	The proof of Theorem~\ref{thm:init-lower-bound} is presented in detail in Appendix \ref{sec:init-error-bound}.
	In both Theorem~\ref{thm:init-upper-bound} and Theorem~\ref{thm:init-lower-bound}, the performance of Algorithm~\ref{alg:init-A} is analyzed with the assumption that $\vv^{(j)}$ and $f_i^*$ in \eqref{eq:v-score-A} can be computed exactly. However, the accurate computation of $f_i^*$ may be impractical due to computational costs. Therefore, we explore the error bounds when the score $\vv^{(j)}$ approximates \eqref{eq:v-score-A} with some degree of error.
	We investigate two types of scoring errors. 
	\begin{itemize}
		\item \textbf{Additive error.} There exists $\epsilon>0$, we have access to an estimated $\tilde f_i^*$ satisfying
		\begin{equation}
			\label{eq:add-error}
			f^*_i-\epsilon\leq\tilde f^*_i\leq f^*_i+\epsilon.
		\end{equation}
		Accordingly, we define:
		\begin{equation}
			\label{eq:add-error-score}
			\tilde v_i^{(j)}=\min_{1\leq j'\leq j-1}\left(\max\left(f_i(\vx^{(0)}_{j'})-\tilde f^*_i ,0\right)\right)=\max\left(\min_{1\leq j'\leq j-1}\left(f_i(\vx^{(0)}_{j'})-\tilde f^*_i \right),0\right).
		\end{equation}
		\item \textbf{Scaling error.} There exists a deterministic oracle $O_v:[N]\times \bR^d\rightarrow \bR$, such that for any $\vx\in\bR^d$ and $i\in[N]$,
		\begin{equation}\label{eq:scaling-error}
			c_1 (f_i(\vx)-f_i^*) \leq O_v(i,\vx)\leq c_2 (f_i(\vx)-f_i^*).
		\end{equation}
		Set
		\begin{equation}
			\label{eq:mult-error}
			\tilde v_i^{(j)}=\min_{1\leq j'\leq j-1}O_v(i,\vx_{j'}^{(0)}).
		\end{equation}
	\end{itemize}
	
	We first analyze the performance of Algorithm~\ref{alg:init-A} using the score $\tilde v_i^{(j)}$ with additive error as in \eqref{eq:add-error-score}. We typically require the assumption that the solution set $S^*$ is $(k,\sqrt\frac{2\epsilon}{\mu})$-separate, which guarantees that
	$$
	\sum_{i=1}^N \min_{j\in[l]}\max\left((f_i(\vz_j)-\tilde f_i^*),0\right)>0,
	$$
	for any $l<k$ and $\vz_1,\vz_2,\ldots,\vz_l\in\bR^d$.
	Hence in the initialization Algorithm \ref{alg:init-A} with score \eqref{eq:add-error-score}, there is at least one $\tilde v_i^{(j)}>0$ in each round. 
	We have the following generalized version of Theorem~\ref{thm:init-upper-bound} with additive error.
	
	\begin{theorem} 
		\label{thm:add-noisy-init-bound}
		Under Assumptions \ref{asp:L-smooth} and \ref{asp:mu-strong-convex}, suppose that we have $\{\Tilde{f}_i^*\}_{i=1}^N$ satisfying \eqref{eq:add-error} for some noise factor $\epsilon>0$, and that the solution set $S^*$ is $\big(k,\sqrt\frac{2\epsilon}{\mu}\big)$-separate.
		Then for the initialization Algorithm~\ref{alg:init-A} with the scores in \eqref{eq:v-score-A} replaced by the noisy scores in
		\eqref{eq:add-error-score}, we have
		\begin{equation}
			\label{eq:noisy-init-bound}
			\bE F(\calM_{\textup{init}}) - f^* \leq
			4(2+\ln k) \left( \frac{L^2}{\mu^2}+\frac{L}{\mu} \right)(F^*-f^*) + \epsilon\cdot \left(1+(2+\ln k) \left(1+\frac{4L}{\mu}\right)\right).
		\end{equation}
	\end{theorem}
	The proof of Theorem~\ref{thm:add-noisy-init-bound} is deferred to Appendix \ref{sec:init-error-bound}. Next, we state a similar result for the scaling-error oracle as in \eqref{eq:mult-error}, whose proof is deferred to Appendix~\ref{sec:init-error-bound}.
	
	\begin{theorem} 
		\label{thm:multi-noisy-init-bound}
		Suppose that Assumptions \ref{asp:L-smooth}--\ref{asp:mu-strong-convex} hold and that the solution set $S^*$ is $k$-separate.
		Then for the initialization Algorithm \ref{alg:init-A} with the scores in \eqref{eq:v-score-A} replaced by the scores in
		\eqref{eq:mult-error},
		we have the following bound:
		\begin{equation*}
			\bE F(\calM_{\textup{init}}) - f^*  \leq 4\left(\frac{c_2}{c_1}\frac{L}{\mu}+\frac{c_2^2}{c_1^2}\frac{L^2}{\mu^2} 
			\right)(2+\ln k)(F^*-f^*).
		\end{equation*}
	\end{theorem}
	
	Recall that we introduce an alternative score in \eqref{eq:v-score-D}. This score can actually be viewed as a noisy version of \eqref{eq:v-score-A} with a scaling error.
	Under Assumptions \ref{asp:L-smooth} and \ref{asp:mu-strong-convex}, it holds that
	$$
	2\mu(f_i(\vx)-f_i^*)\leq\|\nabla f_i(\vx) \|^2\leq 2L(f_i(\vx)-f_i^*),
	$$
	for any $i\in[N]$ and $\vx\in\bR^d$, which satisfies \eqref{eq:scaling-error} with $c_1=2\mu$ and $c_2=2 L$. Therefore, we have a direct corollary of Theorem~\ref{thm:multi-noisy-init-bound}.
	
	\begin{corollary}
		Suppose that Assumptions \ref{asp:L-smooth} and \ref{asp:mu-strong-convex} hold and that the solution set $S^*$ is $k$-separate.
		For the initialization Algorithm \ref{alg:init-A} with the scores in \eqref{eq:v-score-A} replaced by the scores in
		\eqref{eq:v-score-D},
		we have
		\begin{equation*}
			\bE F(\calM_{\textup{init}}) - f^* \leq 4\left(\frac{L^2}{\mu^2}+\frac{L^4}{\mu^4} 
			\right)(2+\ln k)(F^*-f^*).
		\end{equation*}
	\end{corollary}

	\subsection{Convergence rate of Lloyd's algorithm} 
	In this subsection, we state convergence results of Lloyd's Algorithm~\ref{alg:Lloyd} and momentum Lloyd's Algorithm~\ref{alg:Lloyd-momentum}, with all proofs being deferred to Appendix~\ref{sec:conv-Lloyd}. For Algorithm~\ref{alg:Lloyd}, the optimization process of $\vx_j^\tto$ follows a gradient descent scheme on a varying objective function $F_j^\tto$, which is the average of all active $f_i$'s determined by $C_j^{(t)}$ in \eqref{eq:C-partition-update}. We have the following gradient-descent-like convergence rate on the gradient norm $\|\nabla F_j^\tto (\vx_j^\tto)\|$.
	
	\begin{theorem}
		\label{thm:lloyd-gd-conv}
		Suppose that Assumption \ref{asp:L-smooth} is satisfied and we take the step size $\gamma=\frac{1}{L}$ in Algorithm \ref{alg:Lloyd}. Then 
		\begin{equation*}
			\frac{1}{T+1}\sum_{t=0}^T\sum_{j=1}^k  \frac{|C_j^{(t)}|}{N} \left\|\nabla F_j^{(t)}(\vx_j^{(t)}) \right\|^2 \leq \frac{2L}{T+1}\left( F(\vx_1^{(0)},\vx_2^{(0)},\ldots,\vx_k^{(0)})-F^\star\right).
		\end{equation*}
	\end{theorem}
	
	For momentum Lloyd's Algorithm~\ref{alg:Lloyd-momentum}, we have a similar convergence rate stated as follows.
	
	\begin{theorem}
		\label{thm:lloyd-momentum-conv}
		Suppose that Assumption~\ref{asp:L-smooth} holds and that $\alpha>1$. For Algorithm \ref{alg:Lloyd-momentum}, there exists a constant $ \Bar{\gamma}(\alpha,\beta,L)$, such that
		\begin{equation*}
			\avetT\sumjk \frac{|C_j^\tto|}{N}\|\nabla F_j^{(t)}(\vx_j^{(t)})\|^2\leq \frac{2(1-\beta)}{\gamma}\cdot
			\frac{F(\vx_1^{(0)},\vx_2^{(0)},\ldots,\vx_k^{(0)})-F^*}{T},
		\end{equation*}
		as long as $\gamma\leq \Bar{\gamma}(\alpha,\beta,L)$.
	\end{theorem}
	
	\section{Numerical Experiments}
	\label{sec:num-exp}

	In this section, we conduct numerical experiments to demonstrate the efficiency of the proposed model and algorithms. Our code with documentation can be found at \url{https://github.com/LisangDing/Sum-of-Minimum_Optimization}.

	\subsection{Comparison between the sum-of-minimum model and the product formulation}
	\label{sec:compare-product}
	
	We consider two optimization models for generalized principal component analysis: the sum-of-minimum formulation \eqref{eq:GPCA} and another widely acknowledged formulation given by~\cites{peng2023block, vidal2005generalized}:
	\begin{equation}
		\label{eq:GPCA-prod}
		\Min_{\vvA_j^\top \vvA_j = I_r} \frac{1}{N}\sum_{i=1}^N \prod_{j=1}^k\|\vy_i^\top \vvA_j\|^2.
	\end{equation}
	The initialization for both formulations is generated by Algorithm~\ref{alg:init-A}. We use a slightly modified version of Algorithm~\ref{alg:Lloyd} to minimize \eqref{eq:GPCA} since the minimization of the group functions for GPCA admits closed-form solutions. In particular, we alternatively compute the minimizer of each group objective function as the update of $\vA_j$ and then reclassify the sub-functions. We use the block coordinate descent (BCD) method \cite{peng2023block} to minimize \eqref{eq:GPCA-prod}. The BCD algorithm alternatively minimizes $\vvA_j$ with all other $\vvA_l\ (l\neq j)$ being fixed. The pseudo-codes of both algorithms are included in Appendix~\ref{sec:supp-detail-product}.
	
	We set the cluster number $k\in\{2,3,4\}$, dimension $d\in\{4,5,6\}$, subspace co-dimension $r=d-2$, and the number of data points $N=1000$. The generalization of the dataset $\{\bvy_i\}_{i=1}^N$ is described in Appendix~\ref{sec:supp-detail-product}. We set the maximum iteration number as 50 for Algorithm~\ref{alg:Lloyd} with \eqref{eq:GPCA} and terminate the algorithm once the objective function stops decreasing, i.e., the partition/clustering remains unchanged. Meanwhile, we set the iteration number to 50 for the BCD algorithm \cite{peng2023block} with \eqref{eq:GPCA-prod}. The sythetic data generation is elaborated in Appendix~\ref{sec:supp-detail-product}. The classification accuracy of both methods is reported in Table~\ref{tab:GPCA}, where the classification accuracy is defined as the maximal matching accuracy with respect to the ground truth over all permutations. We observe that our model and algorithm lead to significantly higher accuracy. This is because, compared to \eqref{eq:GPCA-prod}, the formulation in \eqref{eq:GPCA} models the requirements more precisely, though it is more difficult to optimize due to the non-smoothness.
	
	\begin{table}[htb!]
		\caption{Cluster accuracy percentages of the sum-of-minimum (SoM) vs. the sum-of-product (SoP) GPCA models after 50 iterations.}
		\centering
		\begin{tabular}{c|c|ccc}
			\toprule
			&  & $d=4$ & $d=5$ & $d=6$ \\ \midrule
			\multirow{2}{*}{$k=2$}  & SoM & \textbf{98.24}        & \textbf{98.07}         & \textbf{98.19}         \\
			& SoP  &  81.88         &  75.90         &  73.33          \\\midrule
			\multirow{2}{*}{$k=3$}  & SoM   & \textbf{95.04}          & \textbf{94.98}         & \textbf{95.94}          \\
			& SoP  & 67.69         & 62.89         &  60.85          \\\midrule
			\multirow{2}{*}{$k=4$}  & SoM  & \textbf{91.30}        & \textbf{92.92}      & \textbf{93.73}          \\ 
			& SoP  & 62.36       &59.65       & 57.89          \\\bottomrule
		\end{tabular}
		
		\label{tab:GPCA}
	\end{table}
	
	Next, we compare the computational cost for our model and algorithms with that of the product model and the BCD algorithm. We observe that the BCD algorithm exhibited limited improvements in accuracy after the initial 10 iterations. Thus, for a fare comparison, we set both the maximum iterations for our model and algorithms and the iteration number for the BCD algorithm to 10. The accuracy rate and the CPU time are shown in Table~\ref{tab:GPCA-time}, from which one can see that the computational costs of our algorithm and the BCD algorithm are competitive, while our algorithm achieves much better classification accuracy.

	\begin{table}[htb!]
		\caption{Averaged cluster accuracy percentage / CPU time in seconds for GPCA after 10 iterations on sum-of-minimum (SoM) and the sum-of-product (SoP)  models.
		}
		\centering
		\begin{tabular}{c|c|cccc}
			\toprule
			& & $d=4$                 & $d=5$                 & $d=6$                 \\ \midrule
			\multirow{2}{*}{$k=2$}  & SoM & \textbf{97.84 / 0.08} & \textbf{97.93 / 0.08} & \textbf{98.01 / 0.08} \\
			& SoP & 81.78 / 0.14          & 75.76 / 0.14          & 73.24 / 0.15          \\ \midrule
			\multirow{2}{*}{$k=3$} & SoM  & \textbf{93.34 / 0.19} & \textbf{94.14 / 0.19} & \textbf{95.25 / 0.16} \\
			& SoP  & 67.18 / 0.20          & 62.76 / 0.22          & 60.80 / 0.20          \\ \midrule
			\multirow{2}{*}{$k=4$} & SoM  & \textbf{88.62 / 0.32} & \textbf{91.78 / 0.29} & \textbf{92.62 / 0.27} \\
			& SoP  & 61.52 / 0.26          & 59.37 / 0.27          & 57.82 / 0.27   \\ \bottomrule           
		\end{tabular}
		
		\label{tab:GPCA-time}
	\end{table}

	\subsection{Comparison between different initializations}
	\label{sec:compare-init}

	We present the performance of Lloyd's Algorithm \ref{alg:Lloyd} combined with different initialization methods. The initialization methods adopted in this subsection are:
	\begin{itemize}
		\item \textbf{Random initialization. } We initialize variables $\vx_1^{(0)},\vx_2^{(0)},\ldots,\vx_k^{(0)}$ with i.i.d. samples from the $d$-dimensional standard Gaussian distribution.
		\item \textbf{Uniform seeding initialization. }We uniformly sample $k$ different indices $i_1,i_2,\ldots,i_k$ from $[N]$, then we set $\vx_{i_j}^*$ as the initial value of $\vx_j^{(0)}$.
		\item \textbf{Proposed initialization. } We sample the $k$ indices using Algorithm \ref{alg:init-A} and initialize $\vx_j^{(0)}$ with the minimizer of the corresponding sub-function.
	\end{itemize}
	\textbf{Mixed linear regression.} Our first example is the $\ell_2$-regularized mixed linear regression. We add an $\ell_2$ regularization on each sub-function $f_i$ in \eqref{eq:mixed-linear-regression} to guarantee strong convexity, and the sum-of-minimum optimization objective function can be written as 
	$$
	\frac{1}{N}\sum_{i=1}^N \min_{j\in[k]}\left\{  \frac{1}{2}\left(g(\va_i;\vx_j)-b_i\right)^2 +\frac{\lambda}{2}\|\vx_j\|^2 \right\},
	$$
	where $\{(\va_i,b_i)\}_{i=1}^N$ collects all data points and $\lambda>0$ is a fixed parameter. The dataset $\{(\va_i,b_i)\}_{i=1}^N$ is generated as described in Appendix~\ref{sec:supp-detail-init}.
	
	Similar to the GPCA problem, we slightly modify Lloyd's algorithm since the $\ell_2$-regularized least-square problem can be solved analytically.
	Specifically, we use the minimizer of the group objective function as the update of $\vx_j$ instead of performing the gradient descent as in \eqref{eq:group_gd} or Algorithm~\ref{alg:Lloyd}. We perform the algorithm until a maximum iteration number is met or the objective function value stops decreasing. The detailed algorithm is given in Appendix \ref{sec:supp-detail-init}.

	In the experiment, the number of samples is set to $N=1000$ and we vary $k$ from 4 to 6 and $d$ (the dimension of $\va_i$ and $\vx_j$) from 4 to 8. For each problem with fixed cluster number and dimension, we repeat the experiment for 1000 times with different random seeds. In each repeated experiment, we record two metrics. If the output objective value at the last iteration is less than or equal to $F(\vx_1^+,\vx_2^+,\ldots,\vx_k^+)$, where $(\vx_1^+,\vx_2^+,\ldots,\vx_k^+)$ is the ground truth that generates the dataset $\{(\va_i,b_i)\}_{i=1}^N$, we consider the objective function to be nearly optimized and label the algorithm as successful on the task; otherwise, we label the algorithm as failed on the task. Additionally, we record the number of iterations the algorithm takes to output a result. 
	The result is displayed in Table~\ref{Table:linear-mix-regression}.
	
	\begin{table*}[htb!]
		\caption{The failing rate vs. the average iteration number of three initialization methods when solving mixed linear regression problems with different cluster numbers and dimensionality. A smaller failure rate and a lower average iteration number indicate better performance. The least failure rate among the three methods is bolded, and the least average iteration number under the same cluster number and dimension settings is underlined.}
		
		\resizebox{\textwidth}{!}{
			\begin{tabular}{c|c|ccccccc}
				\toprule 
				& \textbf{Init. Method} & & \textbf{$d=4$} & \textbf{$d=5$} & \hspace{-2mm}\textbf{$d=6$} & \textbf{$d=7$} & \textbf{$d=8$}  &  \\
				\midrule
				\multirow{3}{*}{$k=4$} & \textit{random}&                        & 0.056 / 17.577          & \textbf{0.031} / 18.378 & 0.038 / 19.923          & 0.058 / 21.631         & 0.071 / 22.344      &   \\ 
				& \textit{unif. seeding}   &                 & 0.057 / 16.139         & 0.034 / 16.885         & 0.050 / 18.022          & 0.055 / 18.708         & 0.075 / 19.959       &  \\ 
				& \textit{proposed} &                    & \textbf{0.050} / \textbf{14.551}  & 0.036 / \textbf{15.276}         & \textbf{0.034} / \textbf{16.020} & \textbf{0.044} / \textbf{16.936} & \textbf{0.051} / \textbf{17.409}& \\ \hline
				\multirow{3}{*}{$k=5$} & \textit{random}  &                      & 0.161 / 26.355         & 0.156 / 28.844         & 0.172 / 32.247         & 0.238 / 35.042         & 0.321 / 38.324      &   \\ 
				& \textit{unif. seeding} &                   & \textbf{0.145} / 23.728 & 0.136 / 25.914         & \textbf{0.143} / 27.671 & 0.198 / 29.935         & 0.256 / 32.662       &  \\ 
				& \textit{proposed}       &              & 0.162 / \textbf{21.552}         & \textbf{0.130} / \textbf{23.476}  & \textbf{0.143} / \textbf{25.933} & \textbf{0.161} / \textbf{27.268} & \textbf{0.217} / \textbf{29.086} &\\ \hline
				\multirow{3}{*}{$k=6$} & \textit{random}&                        & 0.363 / 35.831          & 0.382 / 41.043         & 0.504 / 43.999         & 0.594 / 47.918         & 0.739 / 48.730      &   \\ 
				& \textit{unif. seeding}   &                 & 0.347 / 31.536         & 0.350 / 35.230          & 0.408 / 39.688         & 0.524 / 42.453         & 0.596 / 43.117 &        \\ 
				& \textit{proposed}   &                  & \textbf{0.339} / \textbf{29.610} & \textbf{0.312} / \textbf{33.460} & \textbf{0.389} / \textbf{36.068} & \textbf{0.463} / \textbf{39.010} & \textbf{0.563} / \textbf{40.320}  & \\ \bottomrule
			\end{tabular}
		}
		
		\label{Table:linear-mix-regression}
	\end{table*}

	\textbf{Mixed nonlinear regression. }Our second experiment is on mixed nonlinear regression using 2-layer neural networks. We construct $k$ neural networks with the same structure and let the $j$-th neural network be:
	$$
	\psi(\va;\vW_j,\vp_j,\vq_j,o_j)=\vp_j^\top\textup{ReLU}(\vW_j \va +\vq_j )
	+ o_j.
	$$
	Here, $\va$ is the input data. We let $d_I$ be the input dimension and $d_H$ be the hidden dimension. The dimensions of the variables are $\va\in\bR^{d_I}, \vW_j\in\bR^{d_H\times d_I},\vp_j,\vq_j\in\bR^{d_H},o_j\in\bR.$ We denote $\theta_j=(\vW_j,\vp_j,\vq_j,o_j)$ as the trainable parameters in the neural network. For each trial, we prepare the ground truth $\theta_j^+$ and the dataset $\{(\va_i,b_i)\}_{i=1}^N$ as described in Appendix~\ref{sec:supp-detail-init}.
	We use the squared $\ell_2$ loss for each neural network and construct the $i$-th sub-function as:
	$$
	f_i(\theta)=\frac{1}{2}(\psi(\va_i;\theta)-b_i)^2+\frac{\lambda}{2}\|\theta\|^2,
	$$
	where we still use $\frac{\lambda}{2}\|\theta\|^2, \lambda>0$ as a regularization term. We perform parallel experiments on training the neural networks via Algorithm \ref{alg:Lloyd} using three different initialization methods. During the training process of neural networks, stochastic gradient descent is commonly used to manage limited memory, reduce training loss, and improve generalization. Moreover, the ADAM algorithm proposed in \cite{kingma2014adam} is widely applied. This optimizer is empirically observed to be less sensitive to hyperparameters, more robust, and to converge faster. To align with this practice, we replace the group gradient descent in Algorithm~\ref{alg:Lloyd} and the group momentum method in Algorithm~\ref{alg:Lloyd-momentum} with ADAM optimizer-based backward propagation for the corresponding group objective function. 
	
	We use two metrics to measure the performance of the algorithms. In one set of experiments, we train $k$ neural networks until the value of the loss function $F$ under parameters $\theta_1,\theta_2,\ldots,\theta_k$ is less than that under $\theta_1^+,\theta_2^+,\ldots,\theta_k^+$. We record the average iterations required to achieve the optimization loss. In the other set of experiments, we train $k$ neural networks for a fixed number of iterations. Then, we compute the training and testing loss of the trained neural network, where the training loss on the dataset $\{(\va_i,b_i)\}_{i=1}^N$ is defined as $\frac{1}{N}\sum_{i=1}^N \min_j\left(\frac{1}{2}(\psi(\va_i;\theta_j)-b_i)^2\right)$ and the testing loss is defined in a similar way.

	In our experiments, the training dataset size is $N=1000$ and the testing dataset size is $200$. The testing dataset is generated from the same distribution as the training data. Benefiting from ADAM's robust nature regarding hyperparameters, we use the default ADAM learning rate $\gamma=1e-3$. We set $r=10$ in Lloyd's Algorithm \ref{alg:Lloyd} and fix the cluster number $k=5$. We test on three different $(d_I,d_H)$ tuples: $(5,3)$, $(7,5)$, and $(10,5)$. The results can be found in Table \ref{tab:nonlinear-regression-1} and \ref{tab:nonlinear-regression-2}.
	\begin{table}[htb!]
		\caption{Average epochs for different seeding methods to achieve the ground truth model training loss.}
		\centering
		\begin{tabular}{cccc}
			\toprule
			$(d_I,d_H)$                 & (5,3) & (7,5) & (10,5) \\ \hline
			\textit{random}     & 329.4          & 132.1          & 130.8           \\
			\textit{unif. Seeding}  & 233.1          & 71.2           & 67.6            \\
			\textit{proposed} & \textbf{181.4}          & \textbf{49.3}           & \textbf{47.2}            \\ \bottomrule
		\end{tabular}
		\label{tab:nonlinear-regression-1}
	\end{table}
	
	\begin{table}[htb!]
		\caption{The training errors vs. the testing errors (unit: 10e-3) of Lloyd's algorithm with fixed training epoch numbers.}
		\centering
		\begin{tabular}{cccc}
			\toprule
			$(d_I,d_H)$ / Iter.                   & (5,3) / 300 & (7,5) / 150 & (10,5) / 150 \\ \midrule
			\textit{random}      & 4.26 / 4.63         & 4.57 / 5.54         & 4.62 / 5.82          \\
			\textit{unif. Seeding}   & 3.86 / 4.25         & 3.96 / 4.77         & 3.56 / 4.52          \\
			\textit{proposed}  & \textbf{3.44} / \textbf{3.93}         & \textbf{3.51} / \textbf{4.37}         & \textbf{3.39} / \textbf{4.34}          \\ \bottomrule
		\end{tabular}
		\label{tab:nonlinear-regression-2}
	\end{table}
	
	We can conclude from Table~\ref{Table:linear-mix-regression}, \ref{tab:nonlinear-regression-1}, and \ref{tab:nonlinear-regression-2} that the careful seeding Algorithm~\ref{alg:init-A} generates the best initialization in most cases. This initialization algorithm results in the fewest iterations required by Lloyd's algorithm to converge, the smallest final loss, and the highest probability of finding the ground-truth clustering.

	\section{Conclusion}
	\label{sec:conclude}
	
	This paper proposes a general framework for sum-of-minimum optimization, as well as efficient initialization and optimization algorithms. Theoretically, tight bounds are established for smooth and strongly convex sub-functions $f_i$. Though this work is motivated by classic algorithms for the $k$-means problem, we extend the ideas and theory significantly for a broad family of tasks. Furthermore, the numerical efficiency is validated for generalized principal component analysis and mixed linear and nonlinear regression problems. Future directions include developing algorithms with provable guarantees for non-convex $f_i$ and exploring empirical potentials on large-scale tasks. 
	
	\section*{Acknowledgements} Lisang Ding receives support from Air Force Office of Scientific Research Grants MURI-FA9550-18-1-0502. We thank Liangzu Peng for fruitful discussions on GPCA.
	
	\bibliographystyle{amsxport}
	\bibliography{reference}
	
	\appendix
	\section{Proof of Proposition \ref{prop:S-k-sep}}
	\allowdisplaybreaks
	In this section, we provide a proof of the proposition in Section \ref{sec:preliminary}.
	\begin{proposition}[Restatement of Proposition \ref{prop:S-k-sep}]
		Under Assumption \ref{asp:mu-strong-convex}, if $|S^*|\geq k$, the optimization problem \eqref{eq:sum-of-min-opt} admits finitely many minimizers.
	\end{proposition}
	\begin{proof}
		If $|S^*|=k$, then the only minimizer up to a permutation of indices is $\vx_1,\vx_2,\ldots,\vx_k$, such that 
		$$
		\{\vx_1,\vx_2,\ldots,\vx_k\}=S^*.
		$$
		Next we consider the case where $|S^*|>k$. Let $\calR$ be the set of all minimizers of \eqref{eq:sum-of-min-opt}. Due to the $\mu$-strong convexity of $f_i$, the set $\calR$ is nonempty. Let $\calT$ be the set of all partitions $C_1,C_2,\ldots,C_k$ of $[N]$, such that $C_j\not=\emptyset$ for all $j\in[k]$. The set $\calT$ is finite. Next, we show there is an injection from $\calR$ to $\calT$. For $\vX=(\vx_1,\vx_2,\ldots,\vx_k)\in\calR$, we recurrently define 
		$$
		C_j^{\vX}=\{i\in[N]\,|\,f_i(\vx_j)=\min_{l}(f_i(\vx_{l}))\}\backslash \left(\cup_{1\leq j'\leq j-1}C_{j'}^\vX\right).
		$$
		We claim that all $C_j^\vX$'s are nonempty. Otherwise, if there is an index $j$ such that $C_j^\vX=\emptyset$, we have a $\vz\in S^*\backslash \{\vx_1,\vx_2,\ldots,\vx_k\}$. Replacing the $j$-th parameter $\vx_j$ with $\vz$, we have
		$$
		F(\vx_1,\vx_2,\ldots,\vx_k)>F(\vx_1,\vx_2,\ldots,\vx_{j-1},\vz,\vx_{j+1},\ldots,\vx_k).
		$$
		This contradicts the assumption that $\vX$ is a minimizer of \eqref{eq:sum-of-min-opt}. Thus, $\vX\mapsto (C_1^\vX,C_2^\vX,\ldots,C_k^\vX)$ is a well-defined map from $\calR$ to $\calT$. Consider another $\vY=(\vy_1,\vy_2,\ldots,\vy_k)\in\calR$. If $C_j^\vX=C_j^\vY$ for all $j\in[k]$, due to the $\mu$-strong convexity of $f_i$'s, we have
		$$
		\vy_j=\argmin_\vz \sum_{i\in C_j^\vY}f_i(\vz)=\argmin_\vz \sum_{i\in C_j^\vX}f_i(\vz)=\vx_j,\quad \forall j\in[k].
		$$
		Thus, the map defined above is injective. Overall, $\calR$ is a finite set.
	\end{proof}
	
	\section{Algorithm details}
	\label{sec:alg-detail}
	In this section, we provide the details of the algorithms presented in Section \ref{sec:alg}. 
	
	\subsection{Initialization with alternative scores}
	\label{sec:initialization-D-detail}
	When the score function  $v_i^{(j)}$ is taken as the squared gradient norm as in \eqref{eq:v-score-D}, the pseudo-code of the initialization can be found in Algorithm \ref{alg:init-D}.
	\begin{algorithm}[htb!]
		\caption{Initialization}
		\label{alg:init-D}
		\begin{algorithmic}[1]
			\State Sample $i_1$ uniformly at random from $[N]$ and compute
			\begin{equation*}
				\vx_1^{(0)} = \vx_{i_1}^* = \argmin_\vx f_{i_1}(\vx).
			\end{equation*}
			\For{$j=2,3,\dots,k$}
			\State Compute scores $\vv^{(j)}=\left(v_1^{(j)},v_2^{(j)},\ldots,v_N^{(j)}\right)$ via
			\begin{equation*}
				v^{(j)}_i = \min_{1\leq j'\leq j-1} \left\| \nabla f_i(\vx_{j'}^{(0)}) \right\|^2.
			\end{equation*}
			\State Compute the sampling weights $\vw^{(j)}=\left(w^{(j)}_1, \dots,w^{(j)}_N\right)$ by normalizing $\{v_i^{(j)}\}_{i=1}^N$,
			\begin{equation*}
				w^{(j)}_i=\frac{v_i^{(j)}}{\sum_{i'=1}^Nv_{i'}^{(j)}}.
			\end{equation*}
			\State Sample $i_j\in[N]$ according to the weights $\vw^{(j)}$ and compute
			\begin{equation*}
				\vx_j^{(0)} = \vx_{i_j}^* = \argmin_\vx f_{i_j}(\vx).
			\end{equation*}
			\EndFor
		\end{algorithmic}
	\end{algorithm}

	\subsection{Details on momentum Lloyd's Algorithm}
	In this section, we elaborate on the details of momentum Lloyd's Algorithm \ref{alg:Lloyd-momentum}. We use $\vx_1^\tto,\vx_2^\tto,\ldots,\vx_k^\tto$ as the $k$ variables to be optimized. Correspondingly, we introduce $\vm_1^\tto,\vm_2^\tto,\ldots,\vm_k^\tto$ as their momentum. We use the same notation $F_j^\tto$ in \eqref{eq:F-j-t} as the group objective function. In each iteration, we update $\vx$ using momentum gradient descent and update $\vm$ using the gradient of the group function.
	\begin{align*}
		&\vx_j^{(t+1)}=\vx_j^\tto-\gamma \vm_j^\tto,\\
		&\vm_j^{(t+1)}=\beta \vm_j^\tto+\nabla F_j^{(t+1)}(\vx_j^{(t+1)}).
	\end{align*}
	The update of $C_j^\tto$ in the momentum algorithm is different from the Lloyd's Algorithm \ref{alg:Lloyd}. We introduce an acceleration quantity $$\vu_j^{(t+1)}=\frac{1}{1-\beta}(\vx_j^{(t+1)}-\beta \vx_j^\tto).$$ 
	Each class is then renewed around the center $\vu_j^{(t+1)}$. We update index $i\in [N]$ to the class $C_j^{(t+1)}$ where $f_i(\vu_j^{(t+1)})$ attains the minimum value among all $j\in[k]$.
	To ensure the stability of the momentum accumulation, we further introduce a controlled reclassification method. We set a reclassification factor $\alpha>1$. We update $C_j^\tto$ to $C_j^{(t+1)}$ in the following way to ensure
	$$\frac{1}{\alpha}|C_j^\tto|\leq |C_j^{(t+1)}| \leq \alpha |C_j^\tto|.$$
	The key idea is to carefully reclassify each index one by one until the size of one class breaks the above restriction.
	We construct $C_{j,0}=C_j^\tto, j\in[k]$ as the initialization of the reclassification.
	We randomly, non-repeatedly pick indices $i$ from $[N]$ one by one. For $l$ looping from 1 to $N$, we let $C_{j,l-1}, j\in[k]$ be the classification before the $l$-th random index is picked. Let $i_l$ be the $l$-th index sampled. We reassign $i_l$ to the $j$-th class, such that
	$$
	f_{i_l}(\vu_j^{(t+1)})=\min_{j'\in[k]}f_{i_l}(\vu_{j'}^{(t+1)}).
	$$
	There will be at most two classes changed due to the one-index reassignment. We update the class notations from $C_{j',l-1}$ to $C_{j',l}$ for all $j'\in [N]$. If there is any change between $C_{j',l-1}$ and $C_{j',l}$, we check whether 
	$$
	\frac{1}{\alpha} |C_{j'}^\tto| \leq |C_{j'}|\leq \alpha |C_{j'}^\tto|
	$$
	holds. If the above restriction holds for all $j'\in[N]$, we accept the reclassification and move on to the next index sample. Otherwise, we stop the process and return $C_j^{(t+1)}=C_{j,l-1}, j\in [k]$. If the reclassification trial successfully loops to the last index. We assign $C_j^{(t+1)}=C_{j,N}, j\in [k]$.

	\section{Initialization error bounds}
	\label{sec:init-error-bound}
	In this section, we prove the error bounds of the initialization Algorithms \ref{alg:init-A} and \ref{alg:init-D}. Before our proof, we prepare the following concepts and definitions.
	\begin{definition}
		For any nonempty $C\subset[N]$, we define
		$$
		\Delta_C:=\frac{1}{|C|}\sum_{i\in C}\sum_{i'\in C}\|\vx_i^*-\vx_{i'}^*\|^2.
		$$
	\end{definition}
	\begin{definition}
		Let $\calI\subset [N]$ be an index set, $\calM\subset \bR^d$ be a finite set, we define
		\begin{align*}
			&\calA(\calI,\calM)=\sum_{i\in\calI}\min_{\vz\in\calM}(f_i(\vz)-f_i(\vx_i^*)),\\
			&\calD(\calI,\calM)=\sum_{i\in\calI}\min_{\vz\in\calM}\|\nabla f_i(\vz)\|^2.
		\end{align*}
	\end{definition}
	Under the $\mu$-strong convexity and $L$-smooth Assumptions \ref{asp:L-smooth} and \ref{asp:mu-strong-convex}, we immediately have
	$$
	\frac{1}{2L}\calD(\calI,\calM)\leq \calA(\calI,\calM)\leq\frac{1}{2\mu}\calD(\calI,\calM).
	$$
	Besides, for disjoint index sets $\calI_1,\calI_2$, we have
	\begin{gather*}
		\calA(\calI_1\cup \calI_2,\calM)=\calA(\calI_1,\calM)+\calA( \calI_2,\calM),\\
		\calD(\calI_1\cup \calI_2,\calM)=\calD(\calI_1,\calM)+\calD( \calI_2,\calM).
	\end{gather*}
	
	For the problem \eqref{eq:sum-of-min-opt}, the optimal solution exists due to the strong convexity assumption on $f_i$'s. We pick one set of optimal solutions $\vz_1^*,\vz_2^*,\ldots,\vz_k^*$. We let
	$$
	\calM_{\textup{OPT}}=\{\vz_1^*,\vz_2^*,\ldots,\vz_k^*\}.
	$$
	
	Based on this optimal solutions, we introduce $(A_1,A_2,\ldots,A_k)$ as a partition of $[N]$. $A_j$'s are disjoint with each other and
	$$
	\bigcup_{j\in [k]}A_j =[N].
	$$
	Besides, for all $i\in A_j$, $f_i(\vx)$ attains minimum at $\vz_j^*$ over $\calM_{\textup{OPT}}$,
	$$
	f_i(\vz_j^*)-f_i(\vx_i^*)=\min_{j'\in [k]} \left( f_i(\vz_{j'}^*)-f_i(\vx_i^*) \right).
	$$
	The choice of $\calM_{\textup{OPT}}$ and $(A_1,A_2,\ldots,A_k)$ is not unique. We carefully choose them so that $A_j$ are non-empty for each $j\in[k]$.
	
	\begin{lemma}
		\label{lem:EA-i-uniform}
		Suppose that Assumption \ref{asp:L-smooth} holds. Let $\calI$ be a nonempty index subset of $[N]$ and let $i$ be sampled uniformly at random from $\calI$. We have
		$$
		\bE_i\, \calA(\calI,\{ \vx_i^* \})\leq \frac{L}{2}\Delta_\calI. 
		$$
	\end{lemma}
	\begin{proof}We have the following direct inequality.
		\begin{align*}
			\bE_i\, \calA(\calI,\{ \vx_i^* \})&=\frac{1}{|\calI|}\sum_{i\in\calI} \calA(\calI,\{ \vx_i^* \})\\
			&=\frac{1}{|\calI|}\sum_{i\in\calI}\sum_{i'\in\calI} \left(f_{i'}(\vx_i^*)-f_{i'}(\vx_{i'}^*)\right) \\
			&\leq \frac{1}{|\calI|}\sum_{i\in\calI}\sum_{i'\in\calI} \frac{L}{2}\|\vx_i^*-\vx_{i'}^*\|^2\\
			&=\frac{L}{2}\Delta_\calI.
		\end{align*}
	\end{proof}
	
	\begin{lemma}
		\label{lem:diff-Di-Di'}
		Let $\calM$ be a fixed finite set in $\bR^d$. For two indices $i\not=i'$, we have
		$$
		\calA(\{i\},\calM)\leq \frac{2L}{\mu}\calA(\{i'\},\calM)+L\| \vx_i^*-\vx_{i'}^* \|^2
		$$
	\end{lemma}
	\begin{proof}We have the following inequality.
		\begin{align*}
			\calA(\{i\},\calM)&=\min_{\vz\in\calM} \left(f_i(\vz)-f_i(\vx_i^*)\right)\\
			&\leq \min_{\vz\in\calM} \frac{L}{2}\|\vz-\vx_i^*\|^2\\
			&\leq\min_{\vz\in\calM}L(\|\vz-\vx_{i'}^*\|^2+\|\vx_{i'}^*-\vx_i^*\|^2)\\
			&\leq \frac{2L}{\mu}\min_{z\in\calM} \left(f_{i'}(\vz)-f_{i'}(\vx_{i'}^*)\right) +L\|\vx_{i'}^*-\vx_i^*\|^2\\
			&= \frac{2L}{\mu}\calA(\{i'\},\calM)+L\| \vx_i^*-\vx_{i'}^* \|^2.
		\end{align*}
	\end{proof}
	
	\begin{lemma}
		\label{lem:EA-delta-bound}
		Given an index set $\calI$ and a finite point set $\calM$, suppose that $\calA(\calI,\calM)> 0.$ If we randomly sample an index $i\in\calI$ with probability $\frac{\calA(\{i\},\calM)}{\calA(\calI,\calM)}$, then we have the following inequality,
		$$
		\bE \calA(\calI,\calM\cup\{\vx_i^*\})\leq \left(\frac{L^2}{\mu}+L\right) \Delta_\calI.
		$$
	\end{lemma}
	\begin{proof}We consider the expectation of $\calA(\calI,\calM\cup\{\vx_i^*\})$ over $i\in \calI$. We have the following inequality bound.
		\begin{align*} 
			&\bE \calA(\calI,\calM\cup\{\vx_i^*\})\\
			&\quad=\sum_{i\in\calI}\frac{\calA(\{i\},\calM)}{\calA(\calI,\calM)}\calA(\calI,\calM\cup\{\vx_i^*\})\\
			&\quad=\sum_{i\in\calI}\frac{\calA(\{i\},\calM)}{\calA(\calI,\calM)}\sum_{i'\in\calI}\min(\calA(\{i'\},\calM),f_{i'}(\vx_i^*)-f_{i'}(\vx_{i'}^*))\\
			&\quad\stackrel{(a)}{\leq}
			\sum_{i\in\calI}\frac{\frac{1}{|\calI|}  \sum_{i''\in\calI}\left(  \frac{2L}{\mu} \calA(\{i''\},\calM)+L\|\vx_{i''}^*-\vx_i^*\|^2 \right)  }{\calA(\calI,\calM)}\sum_{i'\in\calI}\min(\calA(\{i'\},\calM), f_{i'}(\vx_i^*)-f_{i'}(\vx_{i'}^*))\\
			&\quad=\frac{2L}{\mu}\frac{1}{|\calI|}\sum_{i\in\calI}\sum_{i'\in\calI} \min(\calA(\{i'\},\calM),f_{i'}(\vx_i^*)-f_{i'}(\vx_{i'}^*))
			\\
			&\quad\quad\quad  +\frac{L }{\calA(\calI,\calM)|\calI|}
			\sum_{i\in\calI} \sum_{i''\in\calI}\| \vx_{i''}^*-\vx_i^* \|^2
			\sum_{i'\in\calI} \min(\calA(\{i'\},\calM),f_{i'}(\vx_i^*)-f_{i'}(\vx_{i'}^*))\\
			&\quad\leq \frac{2L}{\mu}\frac{1}{|\calI|}\sum_{i\in\calI}\sum_{i'\in\calI} \frac{L}{2}\|\vx_i^*-\vx_{i'}^*\|^2
			+\frac{L }{\calA(\calI,\calM)|\calI|}
			\sum_{i\in\calI} \sum_{i''\in\calI}\| \vx_{i''}^*-\vx_i^* \|^2
			\sum_{i'\in\calI} \calA(\{i'\},\calM)\\
			&\quad= \left(\frac{L^2}{\mu}+L\right)\frac{1}{|\calI|}\sum_{i\in\calI}\sum_{i'\in\calI} \|\vx_{i'}^*-\vx_i^*\|^2.
		\end{align*}
		Here, (a) holds when applying Lemma \ref{lem:diff-Di-Di'}.
	\end{proof}

	\begin{lemma}
		\label{lem:delta-A-bound}
		For any $A_l$ in the optimal partition $(A_1,A_2,\ldots,A_k)$, we have
		$$
		\Delta_{A_l}\leq \frac{4}{\mu}\calA(A_l,\calM_{\textup{OPT}}).
		$$
	\end{lemma}
	\begin{proof}
		We let $\bar \vy_l=\frac{1}{|A_l|}\sum_{i\in A_l}\vx_i^*$ be the geometric center of optimal $f_i$ solutions of index set $A_l$.
		\begin{align*}
			\Delta_{A_l}&=\frac{1}{|A_l|}\sum_{i\in A_l}\sum_{i'\in A_l}\|\vx_i^*-\vx_{i'}^*\|^2\\
			&=\frac{1}{|A_l|}\sum_{i\in A_l}\sum_{i'\in A_l}\|\vx_i^*-\bar \vy_l +\bar \vy_l - \vx_{i'}^*\|^2\\
			&=\frac{1}{|A_l|}\sum_{i\in A_l}\sum_{i'\in A_l}\left( \|\vx_i^*-\bar \vy_l\|^2 +\|\bar \vy_l - \vx_{i'}^*\|^2 \right)\\
			&=2\sum_{i\in A_l}\|\vx_i^*-\bar \vy_l\|^2\\
			&=2\min_\vz\sum_{i\in A_l}\|\vx_i^*- \vz\|^2\\
			&\leq \frac{4}{\mu} \min_\vz\sum_{i\in A_l}\left(f_i(\vz)-f_i(\vx_i^*)\right)\\
			&=\frac{4}{\mu} \min_\vz\calA(A_l,\{\vz\})\\
			&=\frac{4}{\mu}\calA(A_l.\calM_{\textup{OPT}}).
		\end{align*}
	\end{proof}

	\begin{proposition}
		\label{prop:ED-bound-D-opt}
		Let $\calI$ be an index set, and $\calM$ be a finite point set. Let $\vz^*$ be a minimizer of the objective function
		$
		\sum_{i\in \calI}\left( f_i(\vz)-f_i(\vx_i^*)\right) $. Suppose that
		$\calA(\calI,\calM)>0.$
		If we sample an index $i\in\calI$ with probability $\frac{\calA(\{i\},\calM)}{\calA(\calI,\calM)}$, then we have the following inequality:
		\begin{equation}
			\label{eq:EA-A-opt-bound}
			\bE \calA(\calI,\calM\cup\{\vx_i^*\})\leq 4 \left(\frac{L^2}{\mu^2}+\frac{L}{\mu}\right) \calA(A_l.\calM_{\textup{OPT}}).
		\end{equation}
		
	\end{proposition}
	\begin{proof}
		The deduction of \eqref{eq:EA-A-opt-bound} is a direct combination of Lemma \ref{lem:EA-delta-bound} and Lemma \ref{lem:delta-A-bound}. 
	\end{proof}
	
	Next we prove that the $\frac{L^2}{\mu^2}$ bound in \eqref{eq:EA-A-opt-bound} is tight. 
	\begin{proposition}
		Fix the dimension $d\geq 1$, there exists an integer $N$. We can construct $N$ $\mu$-strongly convex and $L$-smooth sub-functions $f_1,f_2,\ldots,f_N$, and a finite set $\calM\subseteq\bR^d$. We let $\{f_i\}_{i=1}^N$ be the $N$ sub-functions of the sum-of-minimum optimization problem \eqref{eq:sum-of-min-opt}. When we sample an index $i\in[N]$ with probability $\frac{\calA(\{i\},\calM)}{\calA([N],\calM)}$, we have 
		$$
		\bE \calA([N],\calM\cup\{\vx_i^*\})\geq \frac{L^2}{\mu^2}\calA(A_l,\calM_{\textup{OPT}}).
		$$
	\end{proposition}
	\begin{proof}
		For the cases where the dimension $d\geq 2$, we construct the instance in a more concise way. We consider the following $n+1$ points, $\vx_i^*=(1,0,0,\ldots,0)\in\bR^d$, $i=1,2,\ldots,n$, $\vx_{n+1}^*=(-1,0,0,\ldots,0)\in\bR^d$. All the elements except the first one of $\vx_i^*$ are zero. We construct the following functions $f_i$ with minimizers $\vx_i^*$.
		\begin{equation}
			\label{eq:lb-d-dim}
			\begin{split}
				f_i(y_1,y_2,\ldots,y_d)=
				\frac{L}{2}\left(y_1-1\right)^2+\frac{\mu}{2}\sum_{j=2}^d y_j^2,\quad  i=1,2\ldots,n,\\
				f_i(y_1,y_2,\ldots,y_d)=\frac{\mu}{2}\left( y_1+1 \right)^2+ 
				\frac{L}{2} \sum_{j=2}^d y_j^2,\quad i=n+1.
			\end{split}
		\end{equation}
		We have $f_i^*:=f_i(\vx_i^*)=0$ for all $i\in[n+1]$. We construct the finite set $\calM$ in an orthogonal manner.
		We let $\calM=\{ (0,\xi)  \}$, $\xi\in \bR^{d-1}$ be a single point set. Besides, $\|\xi\|=m\gg 1$. The point $\pmb{\xi}=(0,\xi)$ in $\calM$ is orthogonal to all $\vx_i^*$'s. Consider the expectation over the newly sampled index $i$, we have
		\begin{multline*}
			\bE\sum_{i'=1}^{n+1}\min( f_{i'}(\pmb{\xi})-f_{i'}(\vx_{i'}^*),f_{i'}(\vx_i^*)-f_{i'}(\vx_{i'}^*) )\\ 
			= \frac{n(L+\mu m^2)}{n(L+\mu m^2)+(\mu+Lm^2)}2\mu+\frac{\mu+Lm^2}{n(L+\mu m^2)+(\mu+Lm^2)}2nL.
		\end{multline*}
		We set $m=\exp(n)$. As $n\rightarrow\infty$, we have
		\begin{align*}
			\lim_{n\rightarrow\infty} \bE\sum_{i'=1}^{n+1}\min( f_{i'}(\pmb{\xi})-f_{i'}(\vx_{i'}^*),f_{i'}(\vx_i^*)-f_{i'}(\vx_{i'}^*) )=2\mu+2\frac{L^2}{\mu}.
		\end{align*}
		
		In the meanwhile, we have 
		\begin{gather*}
			\vz^*:=\argmin_\vz \sum_{i'=1}^{n+1}( f_{i'}(\vz)-f_{i'}(\vx_{i'}^*) )
			=\frac{nL-\mu}{nL+\mu},\\
			\sum_{i'=1}^{n+1}( f_{i'}(\vz^*)-f_{i'}(\vx_{i'}^*) )=\frac{2\mu nL}{\mu+nL}\stackrel{n\rightarrow \infty}{\rightarrow}2\mu.
		\end{gather*}
		We have the following error rate:
		$$
		\lim_{n\rightarrow\infty} \frac{\bE\sum_{i'=1}^{n+1}\min( f_{i'}(\pmb{\xi})-f_{i'}(\vx_{i'}^*),f_{i'}(\vx_i^*)-f_{i'}(\vx_{i'}^*) )}{\sum_{i'=1}^{n+1}( f_{i'}(\vz^*)-f_{i'}(\vx_{i'}^*) )}=1+\frac{L^2}{\mu^2}.
		$$

		As for the 1D case, we consider the following $n+1$ points. We let $x_i^*=1, i=1,2,\ldots,n$, and $x_{n+1}^*=0$. We construct:
		\begin{gather*}
			f_i(x)=\left\{
			\begin{split}
				\frac{L}{2}(x-1)^2,\quad x\leq 1,\\
				\frac{\mu}{2}(x-1)^2,\quad x\geq 1,
			\end{split}
			\right.\quad i=1,2,\ldots,n.\\
			f_i(x)=\left\{
			\begin{split}
				&\frac{\mu}{2}x^2,\quad x\leq 1,\\
				&\frac{L}{2}(x-1)^2+\mu\left(x-\frac{1}{2}\right),\quad x\geq 1,
			\end{split}
			\quad i=n+1.
			\right.
		\end{gather*}
		Each $f_i^*$ has the minimizer $x_i^*$. Besides, $f_i^*:=f_i(x_i^*)=0$.
		We let $\calM=\{1+\frac{L}{\mu}\}$ be a single point set. Let $\xi=1+\frac{L}{\mu}$. We have
		\begin{align*}
			& f_i(x_{n+1}^*)-f_i^*=\frac{L}{2},\quad i=1,2,\ldots,n,\\
			& f_{n+1}(x_i^*)-f_{n+1}^*=\frac{\mu}{2},\quad i=1,2,\ldots,n, \\
			& f_i\left(1+\frac{L}{\mu}\right)-f_i^*=\frac{L^2}{2\mu},\quad i=1,2,\ldots,n,\\
			& f_{n+1}\left(1+\frac{L}{\mu}\right)-f_{n+1}^*=\frac{L^3+2\mu^2L+\mu^3}{2\mu^2}.
		\end{align*}
		We have the following expectation:
		\begin{align*}
			\bE\sum_{i'=1}^{n+1}\min( f_{i'}(\xi)-f_{i'}(x_{i'}^*),f_{i'}(x_i^*)-f_{i'}(x_{i'}^*) )&=
			\frac{n\frac{L^2}{2\mu}\cdot\frac{\mu}{2}+\frac{L^3+2\mu^2L+\mu^3}{2\mu^2}\cdot n\frac{L}{2}}{n\frac{L^2}{2\mu}+\frac{L^3+2\mu^2L+\mu^3}{2\mu^2}}\\
			&\stackrel{n\rightarrow\infty}{\rightarrow}\frac{3}{2}\mu+\frac{L^2}{2\mu}+\frac{\mu^2}{2L}.
		\end{align*}

		Besides, we have the minimizer $z^*=\frac{nL}{nL+\mu}$ of the objective function $ \sum_{i=1}^{n+1}( f_{i}(z)-f_{i}(x_{i}^*) )$. We have
		$$
		\sum_{i=1}^{n+1}( f_{i}(z^*)-f_{i}(x_{i}^*) )=\frac{nL\mu}{2(nL+\mu)}\stackrel{n\rightarrow\infty}{\rightarrow}\frac{\mu}{2}.
		$$
		We have the following asymptotic error bound:
		\begin{equation*}
			\lim_{n\rightarrow\infty}
			\frac{\bE\sum_{i=1}^{n+1}\min( f_{i}(\xi)-f_{i}(x_{i}^*),f_{i}(x_i^*)-f_{i}(x_{i}^*) )}{\sum_{i=1}^{n+1}( f_{i}(z^*)-f_{i}(x_{i}^*) )}=3+\frac{L^2}{\mu^2}+\frac{\mu}{L}.
		\end{equation*}

	\end{proof}
	We remark that the orthogonal technique used in the construction of \eqref{eq:lb-d-dim} can be applied in other lower bound constructions in the proofs of the initialization Algorithms \ref{alg:init-A} and \ref{alg:init-D} as well.

	\begin{lemma}
		\label{lem:main}
		We consider the sum-of-minimum optimization \eqref{eq:sum-of-min-opt}. Suppose that $S^*$ is $k$-separate. Suppose that we have fixed indices $i_1,i_2,\ldots,i_j$. We define the finite set
		$
		\calM_j=\{ \vx_{i_1}^*, \vx_{i_2}^*,\ldots,\vx_{i_j}^* \}
		$.
		We define the index sets 
		$
		L_j=\{ l:A_l \cap \{i_1,i_2,\ldots,i_j\}\not= \emptyset \}, L_j^c=\{ l:A_l \cap \{i_1,i_2,\ldots,i_j\}= \emptyset \},
		\calI_j=\cup_{l\in L_j} A_l,\calI_j^c=\cup_{l\in L_j^c} A_l
		$. Let $u=|L_j^c|$. We sample $t\leq u$ new indices. We let
		$
		\calM_{j,s}^+=\{\vx_{i_1}^*,\vx_{i_2}^*,\ldots,\vx_{i_j}^*,\vx_{i_{j+1}}^*,\ldots,\vx_{i_{j+s}}^*\}
		$ for $0\leq s\leq t$.
		In each round of sampling, the probability of $i_{j+s}, s>0$, being sampled as $i$ is $\frac{\calA(\{i\},\calM_{j,s-1}^+)}{\calA([N],\calM_{j,s-1}^+)}$. Then we have the following bound,
		\begin{equation}
			\label{eq:A-bound}
			\bE ~\calA([N],\calM_{j,t}^+)\leq \left( \calA(\calI_j,\calM_j) + 4\left( \frac{L^2}{\mu^2}+\frac{L}{\mu} \right) \calA(\calI_j^c,\calM_{\textup{OPT}}) \right)(1+H_{t})+\frac{u-t}{u}\calA(\calI_j^c,\calM_j).
		\end{equation}
		Here, $H_t=1+\frac{1}{2}+\cdots+\frac{1}{t}$ is the harmonic sum.
	\end{lemma}
	
	\begin{proof}
		We prove by induction on $u=|L_j^c|$ and $t$. We introduce the notation
		$$
		\Phi_j(i)=\calA(\{i\},\calM_j)=\min_{\vz\in\calM_j}(f_i(\vz)-f_i(\vx_i^*)).
		$$
		We show that if \eqref{eq:A-bound} holds for the case $(u-1,t-1)$ and $(u,t-1)$, then it also holds for the case $(u,t)$.
		We first prove two base cases.
		
		\emph{Case 1: $t=0, u> 0$.}
		\begin{equation*}
			\bE~ \calA([N],\calM_{j,t}^+) = \calA([N],\calM_j)= \calA(\calI_j,\calM_j) + \calA(\calI_j^c,\calM_j).
		\end{equation*}
		
		\emph{Case 2: $t=1, u=1$.}
		With probability
		$\frac{\calA(\calI_j,\calM_j)}{\calA([N],\calM_j)}$,
		the newly sampled index $i_{j+1}$ will lie in $\calI_j$, and with probability $\frac{\calA(\calI_j^c,\calM_j)}{\calA([N],\calM_j)}$, it will lie in $\calI_j^c$. 
		We have bounds on the conditional expectation
		\begin{align*}
			\bE\left( \calA([N],\calM_{j,t}^+) \big| i_{j+1}\in \calI_j\right) &\leq \calA([N],\calM_j), \\
			\bE\left( \calA([N],\calM_{j,t}^+) \big| i_{j+1}\in \calI_j^c\right)&=\bE\left( \calA(\calI_j,\calM_{j,t}^+) \big| i_{j+1}\in \calI_j^c\right)+\bE\left( \calA(\calI_j^c,\calM_{j,t}^+) \big| i_{j+1}\in \calI_j^c\right)\\
			&\leq \calA(\calI_j,\calM_j)
			+\sum_{i'\in \calI_j^c}\frac{\Phi_j(i')}{\sum_{i\in \calI_j^c}\Phi_j(i)}  \calA(\calI_j^c,\calM_j\cup\{\vx_{i'}^*\})\\
			&\stackrel{(a)}{\leq} \calA(\calI_j,\calM_j)+\left( 
			\frac{L^2}{\mu}+L \right)\Delta_{\calI_j^c}\\
			&\stackrel{(b)}{\leq} \calA(\calI_j,\calM_j)+4\left( 
			\frac{L^2}{\mu^2}+\frac{L}{\mu} \right) \calA(\calI_j^c,\calM_{\textup{OPT}}) \\
		\end{align*}
		Here, (a) holds when applying Lemma \ref{lem:EA-delta-bound}. (b) holds since $\calI_j^c$ is identical to a certain $A_l$ as $u=1$ and we apply Lemma \ref{lem:delta-A-bound}. Overall, we have the bound:
		\begin{align*}
			\bE \calA([N],\calM_{j,t}^+)&=\frac{\calA(\calI_j,\calM_j)}{\calA([N],\calM_j)} \bE\left( \calA([N],\calM_{j,t}^+) \big| i_{j+1}\in \calI_j \right)
			\\
			&\qquad + \frac{\calA(\calI_j^c,\calM_j)}{\calA([N],\calM_j)} \bE\left( \calA([N],\calM_{j,t}^+) \big| i_{j+1}\in \calI_j^c \right)\\
			&\leq \calA(\calI_j,\calM_j)+4\left( 
			\frac{L^2}{\mu^2}+\frac{L}{\mu} \right) \calA(\calI_j^c,\calM_{\textup{OPT}})+\calA(\calI_j,\calM_j)\\
			&= 2 \calA(\calI_j,\calM_j)+4\left( 
			\frac{L^2}{\mu^2}+\frac{L}{\mu} \right) \calA(\calI_j^c,\calM_{\textup{OPT}})
		\end{align*}
		
		Next, we prove that the case $(u,t)$ holds when the inequality holds for cases $(u-1,t)$ and $(u-1,t-1)$. With probability
		$\frac{\calA(\calI_j,\calM_j)}{\calA([N],\calM_j)}$,
		the first sampled index $i_{j+1}$ will lie in $\calI_j$, and with probability $\frac{\calA(\calI_j^c,\calM_j)}{\calA([N],\calM_j)}$, it will lie in $\calI_j^c$.
		Let
		$$
		\alpha = 4\left( 
		\frac{L^2}{\mu^2}+\frac{L}{\mu} \right).
		$$
		We divide into two cases and compute the corresponding  conditional expectations. For the case where $i_{j+1}$ lies in $\calI_j$, we have the following bound on the conditional expectation.
		\begin{align*}
			&\bE \left( \calA( [N],\calM_{j,t}^+)\,\big|\,i_{j+1}\in \calI_j \right)\\
			&\quad\leq\bE\bigg(\left(
			\calA(\calI_j,\calM_j\cup \{ \vx_{i_{j+1}}^*\})+\alpha \calA(\calI_j^c,\calM_{\textup{OPT}})\right)(1+H_{t-1}) \\
			&\qquad\quad  + \frac{u-t+1}{u} \calA(\calI_j^c,\calM_j\cup \{ \vx_{i_{j+1}}^*\})\big | i_{j+1}\in\calI_j\bigg)\\
			&\quad \leq \left( 
			\calA(\calI_j,\calM_j)+\alpha \calA(\calI_j^c,\calM_{\textup{OPT}})\right)(1+H_{t-1})+  \frac{u-t+1}{u} \calA(\calI_j^c,\calM_j).
		\end{align*}
		For the case where $i_{j+1}$ lies in $\calI_j^c$, we have the following inequality:
		\begin{align*}
			&\bE\, \left( \calA([N],\calM_{j,t}^+ )\,\big| \, i_{j+1}\in\calI_j^c \right)\\
			&\quad\leq \sum_{l\in L_j^c}\frac{\sum_{i\in A_l}\Phi_j(i)}{\sum_{i'\in \calI_j^c}\Phi_j(i')}\bigg[
			\left( 
			\calA(\calI_j\cup A_l,\calM_j\cup\{ \vx_i^*\})+\alpha \calA(\calI_j^c\backslash A_l,\calM_{\textup{OPT}})
			\right)(1+H_{t-1}) \\
			&\quad \quad\quad  +\frac{u-t}{u-1}\calA(\calI_j^c\backslash A_l,\calM_j\cup\{\vx_i^*\})
			\bigg]\\
			&\quad \leq\sum_{l\in L_j^c}\frac{\sum_{i\in A_l}\Phi_j(i)}{\sum_{i'\in \calI_j^c}\Phi_j(i')}
			\bigg[
			\big(
			\calA(\calI_j,\calM_j)+ \calA(A_l,\calM_j\cup \{\vx_i^*\})\\
			&\quad\quad \quad +\alpha(\calA(\calI_j^c,\calM_{\textup{OPT}})- \calA(A_l,\calM_{\textup{OPT}}))
			\big)(1+H_{t-1}) +\frac{u-t}{u-1}\left( \calA(\calI_j^c,\calM_j)-\calA(A_l,\calM_j) \right)
			\bigg]\\
			&\quad \stackrel{(a)}{\leq} \left( \calA(\calI_j,\calM_j)+\alpha \calA(\calI_j^c,\calM_{\textup{OPT}}) \right)(1+H_{t-1})+\frac{u-t}{u-1}\left( \calA(\calI_j^c,\calM_j) -\sum_{l\in L_j^c}\frac{\calA(A_l,\calM_j)^2}{\calA(\calI_j^c,\calM_j)} \right)\\
			&\quad \stackrel{(b)}{\leq} 
			\left( \calA(\calI_j,\calM_j)+\alpha\calA(\calI_j^c,\calM_{\textup{OPT}}) \right)(1+H_{t-1})+\frac{u-t}{u-1}\left( \calA(\calI_j^c,\calM_j) -\frac{1}{u}\calA(\calI_j^c,\calM_j) \right)\\
			&\quad =\left( \calA(\calI_j,\calM_j)+\alpha \calA(\calI_j^c,\calM_{\textup{OPT}}) \right)(1+H_{t-1})+\frac{u-t}{u} \calA(\calI_j^c,\calM_j).
		\end{align*}
		Here, (a) holds when applying Lemma \ref{lem:EA-delta-bound} and Lemma \ref{lem:delta-A-bound}. (b) holds as
		\begin{align*}
			\sum_{l\in L_j^c}\calA(A_l,\calM_j)^2\geq \frac{1}{u}\left(\sum_{l\in L_j^c}\calA(A_l,\calM_j)\right)^2=\frac{1}{u}\calA(\calI_j^c,\calM_j)^2.
		\end{align*}
		Overall, we have the bound:
		\begin{align*}
			&\bE \calA([N],\calM_{j,t}^+)\\
			=&\frac{\calA(\calI_j,\calM_j)}{\calA([N],\calM_j)} \bE\left( \calA([N],\calM_{j,t}^+) \big| i_{j+1}\in \calI_j \right)
			+
			\frac{\calA(\calI_j^c,\calM_j)}{\calA([N],\calM_j)} \bE\left( \calA([N],\calM_{j,t}^+) \big| i_{j+1}\in \calI_j^c \right)\\
			\leq& \left( \calA(\calI_j,\calM_j)+ \alpha \calA(\calI_j^c,\calM_{\textup{OPT}})  \right)(1+H_{t-1})+\frac{u-t}{u}\calA(\calI_j^c,\calM_j)+\frac{1}{u}\frac{\calA(\calI_j^c,\calM_j)\calA(\calI_j,\calM_j)}{\calA([N],\calM_j)}\\
			\stackrel{(a)}{\leq}& \left( \calA(\calI_j,\calM_j)+ \alpha \calA(\calI_j^c,\calM_{\textup{OPT}})  \right)(1+H_{t})+\frac{u-t}{u}\calA(\calI_j^c,\calM_j).
		\end{align*}
		Here, (a) holds since $u\geq t$ and
		$$
		\frac{\calA(\calI_j^c,\calM_j)\calA(\calI_j,\calM_j)}{\calA([N],\calM_j)}\leq \calA(\calI_j,\calM_j).
		$$
		The proof concludes.
	\end{proof}

	\begin{theorem}[Restatement of Theorem \ref{thm:init-upper-bound}]
		\label{thm:init-bound}
		Suppose that the solution set $S^*$ is $k$-separate. Let 
		$$
		\calM_{\textup{init}}=\{ \vx_{i_1}^*,\vx_{i_2}^*,\ldots, \vx_{i_k}^* \}
		$$
		be the initial points sampled by the random initialization Algorithm \ref{alg:init-A}.
		We have the following bound:
		\begin{equation}
			\label{eq:init-bound}
			\bE ~\calA([N],\calM_{\textup{init}})
			\leq  4(2+\ln k) \left( \frac{L^2}{\mu^2}+\frac{L}{\mu} \right) \calA([N],\calM_{\textup{OPT}}).
		\end{equation}
	\end{theorem}
	\begin{proof}
		We start with a fixed index $i_1$, let $\calM_1=\{ \vx_{i_1}^* \}$. Suppose $\vx_{i_1}\in A_l$. Then we use Lemma \ref{lem:main} with $u=k-1, t=k-1$. Let
		$$
		\alpha = 4\left( 
		\frac{L^2}{\mu^2}+\frac{L}{\mu} \right).
		$$
		We have
		\begin{align*}
			\bE ~\calA([N],\calM_{1,k-1}^+)\leq 
			\left( \calA(A_l,\calM_1)+\alpha \calA([N]\backslash A_l,\calM_{\textup{OPT}})  \right)(1+H_{k-1})
		\end{align*}
		The term $\bE \calA([N],\calM_{1,k-1}^+)$ can be regarded as the conditional expectation of $\calA([N],\calM_{\textup{init}})$ given $i_1$
		$$
		\bE ~\calA([N],\calM_{1,k-1}^+)=\bE\left( \calA([N],\calM_{\textup{init}})~|~i_1\right).
		$$
		According to Algorithm \ref{alg:init-A}, the first index $i_1$ is uniformly random in $[N]$. We take the expectation over $i_1$ and get
		\begin{align*}
			&\bE ~\calA([N],\calM_{\textup{init}}) \\
			\leq & \frac{1}{N}\sum_{l\in[k]}\sum_{i\in A_l}
			\left( \calA(A_l,\{\vx_i^*\}) + \alpha\calA([N]\backslash A_l,\calM_{\textup{OPT}}) \right)(1+H_{k-1})\\
			=&\left( \frac{1}{N}\sum_{l\in[k]}\sum_{i\in A_l} \calA(A_l,\{\vx_i^*\})+\alpha\left(\calA([N],\calM_{\textup{OPT}})-\frac{1}{N}\sum_{l\in[k]}|A_l|\calA(A_l,\calM_{\textup{OPT}})\right)  \right)(1+H_{k-1})\\
			\stackrel{(a)}{\leq}& \left(\frac{1}{N}\sum_{l\in[k]}|A_l|\frac{L}{2}\Delta_{A_l}+ 
			\alpha \left(\calA([N],\calM_{\textup{OPT}})-\frac{1}{N}\sum_{l\in[k]}|A_l|\calA(A_l,\calM_{\textup{OPT}})\right)  \right)(1+H_{k-1})\\
			\stackrel{(b)}{\leq}&\left(\frac{1}{N}\sum_{l\in[k]}|A_l|\frac{2L}{\mu}\calA(A_l,\calM_{\textup{OPT}})+ 
			\alpha \left(\calA([N],\calM_{\textup{OPT}})-\frac{1}{N}\sum_{l\in[k]}|A_l|\calA(A_l,\calM_{\textup{OPT}})\right)  \right)\\
			&\qquad\cdot (1+H_{k-1})\\
			\leq& \alpha\calA([N],\calM_{\textup{OPT}})(1+H_{k-1})\\
			\leq& 4(2+\ln k) \left( \frac{L^2}{\mu^2}+\frac{L}{\mu} \right) \calA([N],\calM_{\textup{OPT}}).
		\end{align*}
		Here, (a) holds when applying Lemma \ref{lem:EA-i-uniform}. (b) holds as a result of Lemma \ref{lem:delta-A-bound}.

	\end{proof}

	When we take 
	$$
	f_i(\vx)=\frac{1}{2}\|\vx-\vx_i^*\|^2,
	$$
	the optimization problem \eqref{eq:sum-of-min-opt} reduces to the k-means problem, and Algorithm \ref{alg:init-A} reduces to the k-means++ algorithm. Therefore, according to \cite{arthur2007k}, the bound given in Theorem \ref{thm:init-bound} is \textbf{tight} in $\ln k$ up to a constant. Next, we give a more detailed lower bound considering the conditioning number $\frac{L}{\mu}$.

	\begin{theorem}[Restatement of Theorem \ref{thm:init-lower-bound}]
		Given a fixed cluster number $k>0$, there exists $N>0$. We can construct $N$ $\mu$-strongly convex and $L$-smooth sub-functions $\{f_i\}_{i=1}^N$, whose minimizer set $S^*$ is $k$-separate. Besides, the sum-of-min objective function $F$ satisfies that $F^*>f^*$, so that $\calA([N],\calM_{OPT})>0$. When we apply Algorithm \ref{alg:init-A} to sample the initial centers $\calM_{\textup{init}}$, we have the following error bound:
		\begin{equation}
			\bE\calA([N],\calM_{\textup{init}})\geq  \frac{1}{2}\frac{L^2}{\mu^2}\ln k\calA([N],\calM_{\textup{OPT}}).
		\end{equation}
	\end{theorem}
	\begin{proof}
		We construct the following problem. We fix the cluster number to be $k$. We let the dimension to be $2k$. We pick the vertices of a $k$-simplex as the ``centers" of $k$ clusters. The $k$-simplex is embedded in a $k-1$ dimensional subspace. We let the first $k$ elements of the vertices' coordinates to be non-zero, while the other elements are zero. We denote the first $k$ elements of the $l$-th vertex by $\xi^{(l)}\in\bR^k$. We let the $k$-simplex be centered at the origin, so that the magnitudes $\|\xi^{(l)}\|$'s are the same. We let $m$ be the edge length of the simplex. The functions in each cluster follows the orthogonal construction technique in \eqref{eq:lb-d-dim}. Specifically, in cluster $l$, we construct $n+1$ functions mapping from $\bR^{2k}$ to $\bR$ as
		\begin{equation}
			\label{eq:lb-functions}
			\begin{split}
				f_{i,l}(y)&=\frac{\mu}{2}\|y_{1:d}-\xi^{(l)}\|^2+\frac{\mu}{2}\sum_{j\geq k+1,j\not=k+l}y_j^2+\frac{L}{2}(y_{k+l}+1)^2,\quad i=1,2,\ldots,n,\\
				f_{i,l}(y)&=\frac{L}{2}\|y_{1:d}-\xi^{(l)}\|^2+\frac{L}{2}\sum_{j\geq k+1,j\not=k+l}y_j^2+\frac{\mu}{2}(y_{k+l}-1)^2,\quad i=n+1.
			\end{split}
		\end{equation}
		We have a total of $N=k(n+1)$ sub-functions. We let $m=\exp(n)$, $n\gg 1$, so that $\{f_{i,l}\}_{i=1}^{n+1}$ will be assigned in the same cluster when computing the minimizer of the objective function $F$. We let $\ve_l\in\bR^k$ be the $l$-th unit vector, then the minimizers of the above sub-functions are $\vx_{i,l}^*=[\xi^{(l)};-\ve_{l}] (i=1,2,\ldots,n)$ and $\vx_{n+1,l}^*=[\xi^{(l)};\ve_{l}]$. We let $S^*$ be the set of all the minimizers $\{\vx_{1,l}\}_{l=1}^k\cup \{\vx_{n+1,l}\}_{l=1}^k$.
		For each cluster $l$, we can compute
		$$
		\min_{y}\sum_{i=1}^{n+1}(f_{i,l}(y)-f_{i,l}^*)=\frac{2nL\mu}{nL+\mu}.
		$$
		Thus, we have
		$$
		\calA([N],\calM_{\textup{OPT}})= k\frac{2nL\mu}{nL+\mu}.
		$$
		
		Let $\calM$ be a nonempty subset of $S^*$. We study the optimality gap of $F$ when sampling the new centers based on $\calM$. We divide the $k$ clusters into 4 classes as follows:
		\begin{align*}
			&C_{a}=\{ l\,|\, \vx_{1,l}^*\in \calM, \vx_{n+1,l}^*\not\in\calM \},\\
			&C_{b}=\{ l\,|\, \vx_{n+1,l}^*\in \calM, \vx_{1,l}^*\not\in\calM \},\\
			&C_{f}=\{ l\,|\, \vx_{1,l}^*\in \calM, \vx_{n+1,l}^*\in\calM \},\\
			&C_{u}=\{ l\,|\, \vx_{1,l}^*\not\in \calM, \vx_{n+1,l}^*\not\in\calM \}.
		\end{align*}
		We define $a=|C_a|, b=|C_b|, u=|C_u|$. Consider $\calM$ as the existing centers, we continue sampling $t\leq u$ new centers using Algorithm \ref{alg:init-A}. Let $\vw_1^*,\vw_2^*,\ldots,\vw_t^*$ be the newly sampled centers. We define the quantity
		$$
		\phi_{a,b,u,t}=\bE \calA([N],\calM\cup\{\vw_1^*,\vw_2^*,\ldots,\vw_t^*\}),
		$$
		which is the expected optimality gap after sampling.
		We will prove by induction that
		\begin{equation}
			\label{eq:phi-lb}
			\begin{split}
				\phi_{a,b,u,t}\geq \alpha^{t+1}
				\bigg[\frac{1}{2}
				\left(n\left(\mu m^2 +\mu+L\right)+\left(Lm^2+L+\mu\right)\right)(u-t) \\
				+(2nLb+2\mu a)(1+H_u)+\left( \frac{2L^2}{\mu}+2\mu \right)G_u
				\bigg].
			\end{split}
		\end{equation}
		Here $H_u$ is the harmonic series. $G_u$ is recursively defined as:
		$$
		G_0=0,\quad G_u-G_{u-1}=\beta(1+H_{u-1}).
		$$
		The parameter $0<\alpha,\beta<1$ are chosen as
		$$
		\alpha =1-\frac{1}{m},\quad \beta=1-\frac{1}{\sqrt{n}}.
		$$
		We denote the right0hand side of \eqref{eq:phi-lb} as $\alpha^{t+1}\varphi_{a,b,u,t}$.
		
		We consider the case where $t=0$, we have
		\begin{align*}
			\phi_{a,b,u,0}=\frac{1}{2}
			\left(n\left(\mu m^2 +\mu+L\right)+\left(Lm^2+L+\mu\right)\right)u+2nLb+2\mu a.
		\end{align*}
		In the mean while,
		$$
		\varphi_{a,b,u,0}=\frac{1}{2}
		\left(n\left(\mu m^2 +\mu+L\right)+\left(Lm^2+L+\mu\right)\right)u+(2nLb+2\mu a)(1+H_u)+\left(   \frac{2L^2}{\mu}+2\mu \right)G_u.
		$$
		If $u=0$, we have
		$$
		\phi_{a,b,0,0}= \varphi_{a,b,0,0}\geq \alpha \varphi_{a,b,0,0}.
		$$
		If $u\geq 1$, then $\frac{1}{2}
		\left(n\left(\mu m^2 +\mu+L\right)+\left(Lm^2+L+\mu\right)\right)u$ becomes the leading term,
		\begin{align*}
			&(1-\alpha) \frac{1}{2}
			\left(n\left(\mu m^2 +\mu+L\right)+\left(Lm^2+L+\mu\right)\right)u \\
			&\quad\geq \frac{1}{2}nm\mu u \\
			&\quad\geq (2nLb+2\mu a)(1+H_u)+\left(   \frac{2L^2}{\mu}+2\mu \right)G_u \\
			&\quad\geq \alpha \left((2nLb+2\mu a)(1+H_u)+\left(   \frac{2L^2}{\mu}+2\mu \right)G_u\right).
		\end{align*}
		Rearrange the left-hand side and the right-hand side of the inequality, we have:
		\begin{align*}
			& \frac{1}{2}
			\left(n\left(\mu m^2 +\mu+L\right)+\left(Lm^2+L+\mu\right)\right)u \\
			\geq& \alpha 
			\left(
			\frac{1}{2}
			\left(n\left(\mu m^2 +\mu+L\right)+\left(Lm^2+L+\mu\right)\right)u
			+ (2nLb+2\mu a)(1+H_u)+\left(   \frac{2L^2}{\mu}+2\mu \right)G_u
			\right)\\
			=& \alpha \varphi_{a,b,u,0}.
		\end{align*}
		Therefore, we have
		$$
		\phi_{a,b,u,0}\geq \alpha \varphi_{a,b,u,0}.
		$$
		Next, we induct on $t$. When $t\geq 1$, we have $u\geq 1$. We use the one-step transfer technique. We let 
		\begin{gather*}
			K=\frac{1}{2}
			\left(n\left(\mu m^2 +\mu+L\right)+\left(Lm^2+L+\mu\right)\right)u+2\mu a +2nbL,\\
			A=\frac{1}{2}
			\left(n\left(\mu m^2 +\mu+L\right)+\left(Lm^2+L+\mu\right)\right),\\
			B=\frac{2L^2}{\mu}+2\mu.
		\end{gather*}
		We have
		\begin{align*}
			&\phi_{a,b,u,t}\\
			=&\frac{n(\mu m^2+\mu+L)u}{2K}\phi_{a+1,b,u-1,t-1}+\frac{(Lm^2+L+\mu)u}{2K}\phi_{a,b+1,u-1,t-1} \\
			&+\frac{2nLb}{K}\phi_{a,b-1,u,t-1}+\frac{2\mu a}{K}\phi_{a-1,b,u,t-1}\\
			\geq&\frac{n(\mu m^2+\mu+L)u}{2K}\alpha^t\left[ 
			A(u-t)+(2nLb+2\mu a+2\mu)(1+H_{u-1}) +BG_{u-1} \right]\\
			&+\frac{(Lm^2+L+\mu)u}{2K} \alpha^t \left[A(u-t)+ (2nLb+2\mu a +2nL)(1+H_{u-1})  +BG_{u-1} \right] \\
			&+\frac{2nLb}{K} \alpha^t \left[A(u-t+1)+ (2nLb+2\mu a -2 nL)(1+H_u)  +BG_{u} \right] \\
			&+\frac{2\mu a}{K}\alpha^t \left[A(u-t+1)+ (2nLb+2\mu a-2\mu)(1+H_u)  +BG_{u} \right]\\
			=&\frac{n(\mu m^2+\mu+L)u}{2K} \alpha^t \varphi_{a,b,u,t}\\
			&+\frac{n(\mu m^2+\mu+L)u}{2K}\alpha^t\left[ 2\mu(1+H_{u-1})+(2nLb+2\mu a)(H_{u-1}-H_u)+B(G_{u-1}-G_u) \right]\\
			&+\frac{(Lm^2+L+\mu)u}{2K} \alpha^t \varphi_{a,b,u,t}\\
			&+\frac{(Lm^2+L+\mu)u}{2K} \alpha^t\left[ 2nL(1+H_{u-1})+(2nLb+2\mu a)(H_{u-1}-H_u)+B(G_{u-1}-G_u) \right]\\
			&+\frac{2nLb}{K}\alpha^t\varphi_{a,b,u,t}+\frac{2nLb}{K}\alpha^t(A-2nL (1+H_u))+\frac{2\mu a}{K}\alpha^t\varphi_{a,b,u,t}+\frac{2\mu a}{K}\alpha^t(A-2\mu (1+H_u))\\
			=&\alpha^t\varphi_{a,b,u,t}+\frac{1}{K}\alpha^t\left[ 
			\frac{1}{2}n(\mu m^2+\mu+L)u\left(2\mu-\beta \left( 
			2\mu+\frac{2L^2}{\mu} \right)\right)(1+H_{u-1})\right]\\
			&+\frac{1}{K}\alpha^t\bigg[
			(Lm^2+L+\mu)unL(1+H_{u-1})-\frac{1}{2}(Lm^2+L+\mu)u\beta B(1+H_{u-1})\\
			&\quad\quad\quad\quad -4\mu^2 a(1+H_u)-4n^2L^2b(1+H_u)
			\bigg]\\
			=&\alpha^t\varphi_{a,b,u,t}+\frac{1}{K}\alpha^t\left[ -\frac{1}{2}\left( 
			n(\mu m^2+\mu+L)+(Lm^2+L+\mu) \right)(2nLb+2\mu a) + A(2nLb+2\mu a)\right] \\
			&+\frac{1}{K}\alpha^t\bigg[ 
			\frac{1}{2}n(\mu m^2+\mu+L)u\left(-\frac{2L^2}{\mu}+\frac{1}{\sqrt{n}}(2\mu+\frac{2L^2}{\mu})\right)(1+H_{u-1})\\
			&\quad\quad\quad\quad+(Lm^2+L+\mu)unL(1+H_{u-1})\bigg]\\
			&+\frac{1}{K}\alpha^t\left[
			-\frac{1}{2}(Lm^2+L+\mu)u\beta B(1+H_{u-1})-4\mu^2 a(1+H_u)-4n^2L^2b(1+H_u)
			\right]\\
			=&\alpha^t\varphi_{a,b,u,t}+\frac{1}{K}\alpha^t\sqrt{n} m^2u(1+H_{u-1})(\mu^2+L^2)\\
			&+\frac{1}{K}\alpha^t\bigg[
			n(\mu+L)u\left(L-\frac{L^2}{\mu}+\frac{1}{\sqrt{n}}\left(\mu+\frac{L^2}{\mu}\right)\right)(1+H_{u-1})\\
			&\quad\quad\quad\quad-\frac{1}{2}(Lm^2+L+\mu)u\beta B(1+H_{u-1})-4\mu^2 a(1+H_u)-4n^2L^2b(1+H_u)
			\bigg]\\
			\stackrel{(a)}{\geq}&\alpha^t\varphi_{a,b,u,t}\\
			\geq & \alpha^{t+1}\varphi_{a,b,u,t}.
		\end{align*}
		For (a), we have
		\begin{multline*}
			\sqrt{n} m^2u(1+H_{u-1})(\mu^2+L^2)\geq n(\mu+L)u\left(L-\frac{L^2}{\mu}+\frac{1}{\sqrt{n}}\left(\mu+\frac{L^2}{\mu}\right)\right)(1+H_{u-1})
			\\
			-\frac{1}{2}(Lm^2+L+\mu)u\beta B(1+H_{u-1})-4\mu^2 a(1+H_u)-4n^2L^2b(1+H_u).
		\end{multline*}
		when $m=\exp(n)$ and $n\gg 1$. 
		
		Thus the inequality \eqref{eq:phi-lb} holds. Let $u=t=k-1$. We have
		$$
		\phi_{a,b,k-1,k-1}\geq \alpha^{k}
		\left[(2nLb+2\mu a)(1+H_{k-1})+\left(   \frac{2L^2}{\mu}+2\mu \right)G_{k-1}
		\right].
		$$
		Let $n\geq 100 k^2$. Since $m=\exp(n)\geq 100k^2$, then
		$$
		\alpha^k\geq \frac{3}{4},\quad \beta=1-\frac{1}{10k}\geq \frac{9}{10}.
		$$
		$$
		\phi_{a,b,t-1,t-1}\geq \frac{3}{4}\left( \frac{2L^2}{\mu}+2\mu \right)G_{k-1}.
		$$
		We have the following inequalities:
		\begin{align*}
			H_{k-1}&=1+\frac{1}{2}+\cdots+\frac{1}{k-1}\geq \int_1^k\frac{1}{t}\,dt=\ln k, \quad k\geq 1,\\
			G_k&=\beta\sum_{j=0}^{k-1}(1+H_{j})\geq\beta \left(k +\sum_{j=1}^k\ln j\right)\geq \beta\left(  k+\int_{t=1}^k\ln t\, dt \right)=\beta (k\ln k+1).
		\end{align*}
		Therefore, we have
		$$
		\bE \calA([N],\calM_{\textup{init}})\geq \frac{1}{2} k\ln k \left( \frac{2L^2}{\mu}+2\mu \right)= k\ln k \left( \frac{L^2}{\mu}+\mu \right).
		$$
		In the meanwhile, we have an upper bound estimate for $\calA([N],\calM_{\textup{OPT}})$. We pick $\calM_\xi=\{[\xi^{(l)};-\ve_l]\}_{l=1}^k$ as the centers. We have
		$$
		\calA([N],\calM_{\textup{OPT}})\leq \calA([N],\calM_\xi)= 2k \mu.
		$$
		Thus,
		$$
		\bE\calA([N],\calM_{\textup{init}})\geq \frac{1}{2}\ln k \frac{L^2}{\mu^2}\calA([N],\calM_{\textup{OPT}}).
		$$
	\end{proof}
	
	We prove two different error bounds when the estimate of $f_i(\vz)-f_i(\vx_i^*)$ is not accurate. We consider the additive and multiplicative errors on the oracle $f_i(\vz)-f_i(\vx_i^*)$.
	
	In Algorithm \ref{alg:init-A}, when computing the score $v_i^{(j)}$, we suppose we do not have the exact $f_i^*$, instead, we have an estimate $\tilde f_i^*$, such that
	$$
	|\tilde f_i^*-f_i^*|\leq \epsilon
	$$
	for a certain error factor $\epsilon>0$. We define
	$$
	\tilde\calA(\calI,\calM)=\sum_{i\in\calI}\max\left(\min_{\vz\in\calM}\left(f_i(\vz)-\tilde f_i^*\right),0\right)=
	\sum_{i\in\calI}\min_{\vz\in\calM}\left(\max\left(f_i(\vz)-\tilde f_i^*,0\right)\right).
	$$
	\begin{lemma}
		\label{lem:EA-noisy-bound-A-opt}
		Let $\calI$ be an index set, and $\calM$ be a finite point set. Suppose that $\tilde \calA(\calI,\calM)>0$. We sample an index $i\in\calI$ with probability $\frac{\tilde\calA(\{i\},\calM)}{\tilde\calA(\calI,\calM)}$, then we have the following inequality:
		\begin{equation}
			\label{eq:EA-A-noisy-opt-bound}
			\bE \tilde\calA(\calI,\calM\cup \{\vx_i^*\})\leq |\calI|\left( 1+\frac{4L}{\mu} \right)\epsilon+4\left( \frac{L^2}{\mu^2}+\frac{L}{\mu} \right)\min_\vz\sum_{i\in\calI}(f_i(\vz)-f_i(\vx_i^*)).
		\end{equation}
	\end{lemma}
	\begin{proof}
		We have
		\begin{align*}
			\tilde\calA(\{i\},\calM)&= \max\left(\min_{\vz\in\calM}(f_i(\vz)-\tilde f_i^*) ,0\right)\\
			&\leq \epsilon +\min_{\vz\in\calM}(f_i(\vz)-f_i^*) \\
			&\leq \epsilon+\frac{L}{2}\min_{\vz\in\calM} \|\vz-\vx_i^*\|^2 \\
			&\leq \epsilon+L\|\vx_i^*-\vx_{i'}^*\|^2+L\min_{\vz\in\calM}\|\vz-\vx_{i'}^*\|^2\\
			&\leq \epsilon+L\|\vx_i^*-\vx_{i'}^*\|^2+\frac{2L}{\mu}\min_{\vz\in\calM}(f_{i'}(\vz)-f_{i'}(\vx_{i'}^*))\\
			&\leq \left(1+\frac{2L}{\mu}\right)\epsilon +L\|\vx_i^*-\vx_{i'}^*\|^2+\frac{2L}{\mu} \min_{\vz\in\calM}(f_{i'}(\vz)-\tilde f_{i'}^*) \\
			&\leq \left(1+\frac{2L}{\mu}\right)\epsilon +L\|\vx_i^*-\vx_{i'}^*\|^2+\frac{2L}{\mu}\max\left( \min_{\vz\in\calM}(f_{i'}(\vz)-\tilde f_{i'}^*) ,0\right)\\
			&= \left(1+\frac{2L}{\mu}\right)\epsilon +L\|\vx_i^*-\vx_{i'}^*\|^2+\frac{2L}{\mu}\tilde \calA(\{i'\},\calM).
		\end{align*}
		We have
		\begin{align*}
			&\bE \tilde\calA(\calI,\calM\cup \{\vx_i^*\})\\
			=&\sum_{i\in\calI}\frac{\tilde\calA(\{i\},\calM)}{\tilde \calA(\calI,\calM)}\tilde\calA(\calI,\calM\cup\{\vx_i^* \})\\
			=&\sum_{i\in\calI}\frac{\tilde\calA(\{i\},\calM)}{\tilde \calA(\calI,\calM)}\sum_{i''\in\calI}\tilde\calA(\{i''\},\calM\cup\{\vx_i^* \})\\
			=&\sum_{i\in\calI}\frac{\tilde\calA(\{i\},\calM)}{\tilde \calA(\calI,\calM)}\sum_{i''\in\calI}\min( \tilde\calA(\{i''\},\calM), \max(f_{i''}(\vx_i^*)-\tilde f_{i''}^*,0) )\\
			\leq& \sum_{i\in\calI}\frac{  \frac{1}{|\calI|}\sum_{i'\in\calI}\left\{ \left(1+\frac{2L}{\mu}\right)\epsilon+L\|\vx_i^*-\vx_{i'}^*\|^2+\frac{2L}{\mu}\tilde\calA(\{i'\},\calM)  \right\}  }{\tilde \calA(\calI,\calM)}\\
			&\qquad \cdot \sum_{i''\in\calI}\min( \tilde\calA(\{i''\},\calM),\max(f_{i''}(\vx_i^*)-\tilde f_{i''}^*,0))\\
			\leq & |\calI|\left( 1+\frac{2L}{\mu} \right)\epsilon+L\frac{1}{|\calI|}\sum_{i\in \calI}\sum_{i'\in\calI}\| \vx_i^*-\vx_{i'}^* \|^2+\frac{2L}{\mu}\frac{1}{|\calI|}\sum_{i\in\calI}\sum_{i''\in\calI}\max(f_{i''}(\vx_i^*)-\tilde f_{i''}^*,0)\\
			\leq & |\calI|\left( 1+\frac{4L}{\mu} \right)\epsilon+\left(L+\frac{L^2}{\mu}\right)\frac{1}{|\calI|}\sum_{i\in \calI}\sum_{i'\in\calI}\| \vx_i^*-\vx_{i'}^* \|^2\\
			= & |\calI|\left( 1+\frac{4L}{\mu} \right)\epsilon+2\left(L+\frac{L^2}{\mu}\right)\min_\vz\sum_{i\in \calI}\| \vx_i^*-\vz \|^2\\
			\leq & |\calI|\left( 1+\frac{4L}{\mu} \right)\epsilon+4\left(\frac{L^2}{\mu^2}+\frac{L}{\mu}\right)\min_\vz\sum_{i\in \calI}(f_i(\vz)-f_i(\vx_i^*)).
		\end{align*}
	\end{proof}

	\begin{lemma}
		\label{lem:main-noisy}
		Suppose that we have fixed indices $i_1,i_2,\ldots,i_j$. We define the finite set
		$
		\calM_j=\{ x_{i_1}^*, x_{i_2}^*,\ldots,x_{i_j}^* \}
		$.
		We define the index sets 
		$
		L_j=\{ l: A_l \cap \{i_1,i_2,\ldots,i_j\}\not= \emptyset \}, L_j^c=\{ l:A_l \cap \{i_1,i_2,\ldots,i_j\}= \emptyset \},
		\calI_j=\cup_{l\in L_j}  A_l,\calI_j^c=\cup_{l\in L_j^c}  A_l
		$. Let $u=|L_j^c|$. Suppose that $u>0$. We sample $t\leq u$ new indices. We let
		$
		\calM_{j,s}^+=\{x_{i_1}^*,x_{i_2}^*,\ldots,x_{i_j}^*,x_{i_{j+1}}^*,\ldots,x_{i_{j+s}}^*\}
		$ for $0\leq s\leq t$.
		In each round of sampling, the probability of $i_{j+s}, s>0$, being sampled as $i$ is $\frac{ \tilde\calA(\{i\},\calM_{j,s-1}^+)}{\tilde\calA([N],\calM_{j,s-1}^+)}$. Then we have the following bound:
		\begin{equation}
			\label{eq:A-noisy-bound}
			\begin{split}
				\bE \tilde \calA([N],\calM_{j,t}^+)& \leq (1+H_t) \left[ \tilde \calA(\calI_j,\calM_j) +|\calI_j^c|\left(1+\frac{4L}{\mu}\right)\epsilon +  4\left( \frac{L^2}{\mu^2}+\frac{L}{\mu} \right) \calA(\calI_j^c, \calM_{\textup{OPT}}) \right]
				\\ & + \frac{u-t}{u} \tilde \calA(\calI_j^c,\calM_j).
			\end{split}
		\end{equation}
		
	\end{lemma}
	\begin{proof}
		The key idea of the proof is similar to Lemma \ref{lem:main}.
		We let
		$$
		\alpha = 1+\frac{4L}{\mu},\quad \beta=4\left( \frac{L^2}{\mu^2}+\frac{L}{\mu} \right).
		$$
		We prove by induction.
		When $t=0$, the inequality obviously holds. When $t>0, u=1$, we have the inequality:
		\begin{align*}
			& \bE \tilde\calA([N],\calM_{j,t}^+)\\
			\leq& \frac{\tilde \calA(\calI_j,\calM_j)}{\tilde \calA([N],\calM_j)}\tilde \calA([N],\calM_j)+
			\frac{\tilde \calA(\calI_j^c,\calM_j)}{\tilde \calA([N],\calM_j)} \left( \tilde \calA(\calI_j,\calM_j) +|\calI_j^c|\alpha\epsilon +\beta \calA(\calI_j^c, \calM_{\textup{OPT}}) \right)\\
			\leq& \tilde \calA(\calI_j,\calM_j) +|\calI_j^c|\alpha\epsilon +\beta \calA(\calI_j^c, \calM_{\textup{OPT}})+ \tilde \calA(\calI_j^c,\calM_j).
		\end{align*}
		For the general $(t,u)$ case, $\bE\tilde\calA([N],\calM_{j,t}^+)$ can be bounded by two parts. With probability $\frac{\tilde\calA(\calI_j,\calM_j)}{\tilde\calA([N],\calM_j)}$, the first sampled index lies in $\calI_j$, and the conditional expectation is bounded by:
		$$
		(1+H_{t-1}) \left[ \tilde \calA(\calI_j,\calM_j) +|\calI_j^c|\alpha\epsilon + \beta \calA(\calI_j^c, \calM_{\textup{OPT}}) \right]
		+ \frac{u-t+1}{u} \tilde \calA(\calI_j^c,\calM_j).
		$$
		With probability $\frac{\tilde\calA(\calI_j^c,\calM_j)}{\tilde\calA([N],\calM_j)}$, the first sampled index lies in $\calI_j^c$. 
		The conditional expectation is bounded by:
		\begin{align*}
			&\sum_{l\in L^c}\frac{\tilde\calA( A_l,\calM_j)}{\tilde\calA(\calI_j^c,\calM_j)}\sum_{i\in  A_l}\frac{\tilde\calA (\{i\},\calM_j)}{\tilde \calA( A_l,\calM_j)}\bigg\{(1+H_{t-1})\Big( \tilde \calA(\calI_j\cup A_l,\calM_j\cup\{ \vx_i^*\})+ |\calI_j^c\backslash  A_l|\alpha \epsilon\\
			&\quad\quad  + \beta\calA(\calI_j^c\backslash  A_l, \calM_{\textup{OPT}} ) \Big) +\frac{u-t}{u-1}\tilde\calA(\calI_j^c\backslash  A_l,\calM_j\cup\{ \vx_i^* \})\bigg\}\\
			&\leq \sum_{l\in L^c}\frac{\tilde\calA( A_l,\calM_j)}{\tilde\calA(\calI_j^c,\calM_j)}  \sum_{i\in  A_l}\frac{\tilde\calA (\{i\},\calM_j)}{\tilde \calA( A_l,\calM_j)}
			\bigg\{ (1+H_{t-1}) \Big( \tilde\calA(\calI_j,\calM_j)+ \tilde\calA(A_l,\calM_j\cup \{\vx_i^*\})+ |\calI_j^c\backslash  A_l|\alpha\epsilon\\
			&\quad\quad +\beta \calA(\calI_j^c, \calM_{\textup{OPT}})-\beta\calA(A_l,\calM_{\textup{OPT}}) \Big)+\frac{u-t}{u-1}(\tilde\calA(\calI_j^c ,\calM_j)-\tilde \calA( A_l,\calM_j))\bigg\}\\
			&\leq  (1+H_{t-1}) \left( \tilde\calA(\calI_j,\calM_j)+ |\calI_j^c| \alpha\epsilon+\beta \calA(\calI_j^c, \calM_{\textup{OPT}}) \right)+ \frac{u-t}{u}\tilde\calA(\calI_j^c ,\calM_j).
		\end{align*}
		Overall, we have the following inequality:
		\begin{align*}
			&\bE\tilde\calA([N],\calM_{j,t}^+)\\
			\leq & \frac{\tilde\calA(\calI_j,\calM_j)}{\tilde\calA([N],\calM_j)}\left\{(1+H_{t-1}) \left[ \tilde \calA(\calI_j,\calM_j) +|\calI_j^c|\alpha\epsilon + \beta \calA(\calI_j^c, \calM_{\textup{OPT}}) \right]
			+ \frac{u-t+1}{u} \tilde \calA(\calI_j^c,\calM_j)\right\}\\
			&\quad +\frac{\tilde\calA(\calI_j^c,\calM_j)}{\tilde\calA([N],\calM_j)}
			\left\{  (1+H_{t-1}) \left( \tilde\calA(\calI_j,\calM_j)+ |\calI_j^c| \alpha\epsilon+\beta \calA(\calI_j^c, \calM_{\textup{OPT}}) \right)+ \frac{u-t}{u}\tilde\calA(\calI_j^c ,\calM_j) \right\}\\
			\leq & (1+H_t) \left( \tilde\calA(\calI_j,\calM_j)+ |\calI_j^c| \alpha\epsilon+\beta \calA(\calI_j^c, \calM_{\textup{OPT}}) \right) +\frac{u-t}{u} \tilde\calA(\calI_j^c ,\calM_j).
		\end{align*}
	\end{proof}
	
	\begin{theorem}[Restatement of Theorem \ref{thm:add-noisy-init-bound}]
		\label{thm:noisy-init-bound}
		Suppose that the solution set $S^*$ is $(k,\sqrt{\frac{2\epsilon}{\mu}})$-separate. Let 
		$$
		\calM_{\textup{init}}=\{ \vx_{i_1}^*,\vx_{i_2}^*,\ldots, \vx_{i_k}^* \}
		$$
		be the initial points sampled by the random initialization Algorithm \ref{alg:init-A} with noisy oracles $\tilde f_i^*$.
		We have the following bound:
		\begin{equation*}
			\frac{1}{N}\bE \calA([N],\calM_{\textup{init}})\leq \epsilon+(2+\ln k) \left(1+\frac{4L}{\mu}\right) \epsilon + 4(2+\ln k)\left( \frac{L^2}{\mu^2}+\frac{L}{\mu} \right)\frac{1}{N}\calA([N],\calM_{\textup{OPT}}).
		\end{equation*}
	\end{theorem}
	\begin{proof}
		The proof is similar to that of Theorem \ref{thm:init-bound}. 
		We let
		$$
		\alpha = 1+\frac{4L}{\mu},\quad \beta=4\left(\frac{L^2}{\mu^2}+\frac{L}{\mu}\right).
		$$
		We fix the first index $i_1$. Suppose that $i_1$ lies in $A_l$, we have
		\begin{align*}
			\bE ~\tilde\calA([N],\calM_{1,k-1}^+)\leq 
			\left( \tilde\calA(A_l,\{\vx_{i_1}^*\})+|[N]\backslash A_l|\alpha\epsilon +\beta \calA([N]\backslash A_l,\calM_{\textup{OPT}})  \right)(1+H_{k-1}).
		\end{align*}
		We have
		\begin{align*}
			\bE ~\tilde\calA([N],\calM_{\textup{init}})&\leq \left( \frac{1}{N}\sum_{l\in[k]}\sum_{i\in A_l} \tilde \calA(A_l,\{\vx_i^*\})+N\alpha\epsilon-\frac{1}{N}\sum_{l\in [k]}|A_l|^2\alpha\epsilon\right.\\
			&\quad\quad  \left. +\beta\left(\calA([N],\calM_{\textup{OPT}})-\frac{1}{N}\sum_{l\in[k]}|A_l|\calA(A_l,\calM_{\textup{OPT}})\right)  \right)(1+H_{k-1})\\
			&\stackrel{(a)}{\leq} \left( \frac{1}{N}\sum_{l\in[k]}
			\left(|A_l|^2\epsilon+\frac{L}{2}|A_l| \Delta_{A_l} \right)
			+N\alpha\epsilon-\frac{1}{N}\sum_{l\in [k]}|A_l|^2\alpha\epsilon\right.\\
			&\quad\quad  \left. + \beta\left(\calA([N],\calM_{\textup{OPT}})-\frac{1}{N}\sum_{l\in[k]}|A_l|\calA(A_l,\calM_{\textup{OPT}})\right)  \right)(1+H_{k-1})\\
			&\stackrel{(b)}{\leq} \left( \frac{1}{N}\sum_{l\in[k]}
			\left(|A_l|^2\epsilon+\frac{2L}{\mu}|A_l|\calA(A_l,\calM_{\textup{OPT}}) \right)
			+N\alpha\epsilon-\frac{1}{N}\sum_{l\in [k]}|A_l|^2\alpha\epsilon\right.\\
			&\quad \quad \left. + \beta\left(\calA([N],\calM_{\textup{OPT}})-\frac{1}{N}\sum_{l\in[k]}|A_l|\calA(A_l,\calM_{\textup{OPT}})\right)  \right)(1+H_{k-1})\\
			&\leq \left( N\alpha\epsilon+\beta\calA([N],\calM_{\textup{OPT}}) \right)(1+H_{k-1})\\
			&\leq (2+\ln k) \left(1+\frac{4L}{\mu}\right) N\epsilon + 4(2+\ln k)\left( \frac{L^2}{\mu^2}+\frac{L}{\mu} \right)\calA([N],\calM_{\textup{OPT}}).
		\end{align*}
		Here, (a) holds when applying Lemma \ref{lem:EA-i-uniform}. (b) holds when applying
		$$
		\frac{L}{2}\Delta_{A_l}= L\min_\vz\sum_{i\in A_l} \|\vx_i^*-\vz\|^2\leq \frac{2L}{\mu}\min_\vz\sum_{i\in A_l}(f_i(\vz)-f_i^*)=\frac{2L}{\mu}\calA(A_l,\calM^{(D)}_{\textup{OPT}}).
		$$
		Therefore, we have
		\begin{align*}
			\bE \calA([N],\calM_{\textup{init}})\leq N\epsilon+(2+\ln k) \left(1+\frac{4L}{\mu}\right) N\epsilon + 4(2+\ln k)\left( \frac{L^2}{\mu^2}+\frac{L}{\mu} \right)\calA([N],\calM_{\textup{OPT}}),\\
			\frac{1}{N}\bE \calA([N],\calM_{\textup{init}})\leq \epsilon+(2+\ln k) \left(1+\frac{4L}{\mu}\right) \epsilon + 4(2+\ln k)\left( \frac{L^2}{\mu^2}+\frac{L}{\mu} \right)\frac{1}{N}\calA([N],\calM_{\textup{OPT}}).
		\end{align*}

	\end{proof}
	
	The proof of Theorem \ref{thm:multi-noisy-init-bound} is similar to the proof of Theorem \ref{thm:add-noisy-init-bound}, we skip the details here.
	
	\section{Convergence of Lloyd's algorithm}
	\label{sec:conv-Lloyd}
	In this section, we provide a convergence analysis for Algorithms \ref{alg:Lloyd} and \ref{alg:Lloyd-momentum}.
	
	\begin{theorem}[Restatement of Theorem \ref{thm:lloyd-gd-conv}]
		In Algorithm \ref{alg:Lloyd}, we take the step size $\gamma=\frac{1}{L}$. If $f_i$ are $L$-smooth, we have the following convergence result:
		$$
		\frac{1}{T+1}\sum_{t=0}^T\sum_{j=1}^k  \frac{|C_j^{(t)}|}{N} \|\nabla F_j^{(t)}(\vx_j^{(t)}) \|^2\leq \frac{2L}{T+1}\left( F(\vx_1^{(0)},\vx_2^{(0)},\ldots,\vx_k^{(0)})-F^\star\right).
		$$
		Here, $F^\star$ is the minimum of $F$.
	\end{theorem}
	
	\begin{proof}
		According to the $L$-smoothness assumption on $f_i$, $F_j^{(t)}$ is also $L$-smooth, which implies that
		\begin{align*}
			F_j^{(t)}(\vx_j^{(t+1)})&\leq F_j^{(t)}(\vx_j^{(t)})+\langle \nabla F_j^{(t)}(\vx_j^{(t)}) ,\vx_j^{(t+1)}-\vx_j^{(t)}\rangle+\frac{L}{2}\|\vx_j^{(t+1)}-\vx_j^{(t)}\|^2\\
			&= F_j^{(t)}(\vx_j^{(t)})-\frac{1}{2L}\|  \nabla F_j^{(t)}(\vx_j^{(t)}) \|^2,\\
			\frac{1}{2L}\|  \nabla F_j^{(t)}(\vx_j^{(t)}) \|^2&\leq F_j^{(t)}(\vx_j^{(t)})-F_j^{(t)}(\vx_j^{(t+1)}),\\
			\|  \nabla F_j^{(t)}(\vx_j^{(t)}) \|^2&\leq 2L \left(F_j^{(t)}(\vx_j^{(t)})-F_j^{(t)}(\vx_j^{(t+1)})\right).
		\end{align*}
		Averaging over $\|  \nabla F_j^{(t)}(\vx_j^{(t)}) \|^2$ with weights $|C_j^{(t)}|/N$, we have
		\begin{align*}
			\sum_{j=1}^k  \frac{|C_j^{(t)}|}{N} \|\nabla F_j^{(t)}(\vx_j^{(t)}) \|^2&\leq
			2L\left( F(\vx_1^{(t)},\vx_2^{(t)},\ldots,\vx_k^{(t)})-F(\vx_1^{(t+1)},\vx_2^{(t+1)},\ldots,\vx_k^{(t+1)})\right).
		\end{align*}
		Averaging over $t$ from $0$ to $T$, we have
		\begin{equation*}
			\frac{1}{T+1}\sum_{t=0}^T\sum_{j=1}^k  \frac{|C_j^{(t)}|}{N} \|\nabla F_j^{(t)}(\vx_j^{(t)}) \|^2\leq \frac{2L}{T+1}\left( F(\vx_1^{(0)},\vx_2^{(0)},\ldots,\vx_k^{(0)})-F^\star\right).
		\end{equation*}
	\end{proof}

	Next, we present a convergence theorem for the momentum algorithm. For simplification, we use the notation
	$$
	\vU^{(t)}=(\vu_1^{(t)},\vu_2^{(t)},\ldots,\vu_k^{(t)}).
	$$
	We have the following convergence theorem:
	\begin{theorem}[Restatement of Theorem \ref{thm:lloyd-momentum-conv}]
		Consider Algorithm \ref{alg:Lloyd-momentum}. Suppose that Assumption \ref{asp:L-smooth} holds, $\alpha>1$, and
		$$
		\gamma\leq \min\left( \frac{1-\beta}{2L},\frac{(1-\beta)^\frac{3}{2}(1-\alpha\beta)^\frac{1}{2}}{2\alpha^\frac{1}{2}L\beta} \right).
		$$
		Then it holds that
		$$
		\avetT\sumjk \frac{|C_j^\tto|}{N}\|\nabla F_j^{(t)}(\vx_j^{(t)})\|^2
		\leq \frac{2(1-\beta)}{\gamma}\cdot
		\frac{F(\vx_1^{(0)},\vx_2^{(0)},\ldots,\vx_k^{(0)})-F^*}{T}.
		$$
	\end{theorem}
	
	\begin{proof}
		
		The variable $\vu_j^{(t)}$ satisfies the following property,
		\begin{align*}
			\vu_j^{(t+1)}-\vu_j^{(t)}&=\frac{1}{1-\beta}\left( (\vx_j^{(t+1)}-\vx_j^{(t)})-\beta(\vx_j^{(t)}-\vx_j^{(t-1)}) \right)\\
			&=\frac{-\gamma}{1-\beta}\left(\vm_j^{(t)}-\beta\vm_j^{(t-1)}\right)\\
			&=\frac{-\gamma}{1-\beta}\nabla F_j^{(t)}(\vx_j^{(t)}).
		\end{align*}
		We have the following inequality:
		\begin{align*}
			& F_j^{(t)}(\vu_j^{(t+1)})\\
			\leq&  F_j^{(t)}(\vu_j^{(t)})+\langle \nabla F_j^{(t)}(\vu_j^{(t)}),\vu_j^{(t+1)}-\vu_j^{(t)}  \rangle+\frac{L}{2}\|\vu_j^{(t+1)}-\vu_j^{(t)}\|^2\\
			=&F_j^{(t)}(\vu_j^{(t)})-\frac{\gamma}{1-\beta}\langle \nabla F_j^{(t)}(\vu_j^{(t)}),\nabla F_j^{(t)}(\vx_j^{(t)})  \rangle+\frac{L}{2} 
			\frac{\gamma^2}{(1-\beta)^2} \|\nabla F_j^{(t)}(\vx_j^{(t)})\|^2\\
			=&F_j^{(t)}(\vu_j^{(t)})-\frac{\gamma}{1-\beta}\langle \nabla F_j^{(t)}(\vu_j^{(t)})-\nabla F_j^{(t)}(\vx_j^{(t)}),\nabla F_j^{(t)}(\vx_j^{(t)})  \rangle\\
			&\quad\quad+\left( \frac{L}{2} 
			\frac{\gamma^2}{(1-\beta)^2}-\frac{\gamma}{1-\beta} \right) \|\nabla F_j^{(t)}(\vx_j^{(t)})\|^2\\
			\leq& F_j^{(t)}(\vu_j^{(t)})+\left( \frac{L}{2} 
			\frac{\gamma^2}{(1-\beta)^2}-\frac{\gamma}{1-\beta} \right) \|\nabla F_j^{(t)}(\vx_j^{(t)})\|^2\\
			&\quad\quad +\frac{\gamma}{1-\beta}\frac{\epsilon}{2}\|\nabla F_j^{(t)}(\vu_j^{(t)})-\nabla F_j^{(t)}(\vx_j^{(t)})\|^2+\frac{\gamma}{1-\beta}\frac{1}{2\epsilon}\|\nabla F_j^{(t)}(\vx_j^{(t)})\|^2\\
			\leq& F_j^{(t)}(\vu_j^{(t)})+\left( \frac{L}{2} 
			\frac{\gamma^2}{(1-\beta)^2}-\frac{\gamma}{1-\beta} +\frac{1}{2\epsilon} \frac{\gamma}{1-\beta}\right) \|\nabla F_j^{(t)}(\vx_j^{(t)})\|^2+\frac{\epsilon}{2}\frac{L^2\beta^2\gamma^3}{(1-\beta)^3}\|\vm_j^{(t-1)}\|^2.
		\end{align*}
		Rearranging the inequality, we have
		\begin{multline*}
			\left(\frac{\gamma}{1-\beta}- \frac{L}{2} 
			\frac{\gamma^2}{(1-\beta)^2} -\frac{1}{2\epsilon}\frac{\gamma}{1-\beta} \right) \|\nabla F_j^{(t)}(\vx_j^{(t)})\|^2
			\\ \leq
			F_j^{(t)}(\vu_j^{(t)})-F_j^{(t)}(\vu_j^{(t+1)})+\frac{\epsilon}{2}\frac{L^2\beta^2\gamma^3}{(1-\beta)^3}\|\vm_j^{(t-1)}\|^2.
		\end{multline*}
		We sum over $j=1,2,\ldots,k$ with weights $\frac{|C_j|}{N}$ and get
		\begin{multline*}
			\left(\frac{\gamma}{1-\beta}- \frac{L}{2} 
			\frac{\gamma^2}{(1-\beta)^2} -\frac{1}{2\epsilon}\frac{\gamma}{1-\beta} \right)\sum_{j=1}^k \frac{|C_j^\tto|}{N} \|\nabla F_j^{(t)}(\vx_j^{(t)})\|^2
			\\
			\leq F(\vU^{(t)})-\sumjk \frac{|C_j^\tto|}{N} F_j^{(t)}(\vu_j^{(t+1)})+\frac{\epsilon}{2}\frac{L^2\beta^2\gamma^3}{(1-\beta)^3}\sum_{j=1}^k \frac{|C_j^\tto|}{N} \|\vm_j^{(t-1)}\|^2.
		\end{multline*}
		Since
		$$
		\sumjk \frac{|C_j^\tto|}{N} F_j^{(t)}(\vu_j^{(t+1)})=\frac{1}{N}\sumjk \sum_{i\in C_j^\tto}f_i(\vu_j^{(t+1)})\geq F(\vU^{(t+1)}),
		$$
		we have
		\begin{multline*}
			\left(\frac{\gamma}{1-\beta}- \frac{L}{2} 
			\frac{\gamma^2}{(1-\beta)^2} -\frac{1}{2\epsilon}\frac{\gamma}{1-\beta} \right)\sum_{j=1}^k \frac{|C_j^\tto|}{N} \|\nabla F_j^{(t)}(\vx_j^{(t)})\|^2
			\\
			\leq F(\vU^\tto)-F(\vU^{(t+1)})+\frac{\epsilon}{2}\frac{\alpha L^2\beta^2\gamma^3}{(1-\beta)^3}\sum_{j=1}^k \frac{|C_j^{(t-1)}|}{N} \|\vm_j^{(t-1)}\|^2.
		\end{multline*}
		Summing both sides from $t=1$ to $T$, then dividing both sides by $T$, we have
		\begin{equation}
			\label{eq:momentum-ineq-1}
			\begin{split}
				\left(\frac{\gamma}{1-\beta}- \frac{L}{2} \frac{\gamma^2}{(1-\beta)^2} -\frac{1}{2\epsilon}\frac{\gamma}{1-\beta} \right) \avetT\sum_{j=1}^k \frac{|C_j^\tto|}{N} \|\nabla F_j^{(t)}(\vx_j^{(t)})\|^2 \\
				\leq \frac{F(\vU^{(1)})-F(\vU^{(T+1)})}{T}+\frac{\epsilon}{2}\frac{\alpha L^2\beta^2\gamma^3}{(1-\beta)^3} \avetT \sumjk \frac{|C_j^{(t-1)}|}{N} \|\vm_j^{(t-1)}\|^2.
			\end{split}
		\end{equation}
		Now, we consider the average term $\avetT \frac{|C_j^\tto|}{N}\|\vm_j^{(t)}\|^2$. For $\vm_j^{(t)}$, we have
		\begin{align*}
			\vm_j^{(t)}&=\beta\vm_j^{(t-1)}+\nabla F_j^{(t)}(\vx_j^{(t)})\\
			&=\beta^t\vm_j^{(0)}+\sum_{l=0}^{t-1}\beta^l\nabla F_j^{(t-l)}(\vx_j^{(t-l)})\\
			&=\sum_{l=1}^t\beta^{t-l}\nabla F_j^{(l)}(\vx_j^{(l)}).
		\end{align*}
		We have the following bound on the squared norm of $\vm_j^{(t)}$:
		\begin{align*}
			\|\vm_j^{(t)}\|^2&=\left\|\sum_{l=1}^t\beta^{t-l}\nabla F_j^{(l)}(\vx_j^{(l)})\right\|^2\\
			&=\left(\sum_{s=1}^t\beta^{t-s}\right)^2\left\|\sum_{l=1}^t\frac{\beta^{t-l}}{\sum_{s=1}^t\beta^{t-s}}\nabla F_j^{(l)}(\vx_j^{(l)})\right\|^2\\
			&\stackrel{(a)}{\leq}
			\left(\sum_{s=1}^t\beta^{t-s}\right)^2\sum_{l=1}^t\frac{\beta^{t-l}}{\sum_{s=1}^t\beta^{t-s}}\left\|\nabla F_j^{(l)}(\vx_j^{(l)})\right\|^2\\
			&\leq \frac{1}{1-\beta}\sum_{l=1}^t \beta^{t-l}\left\|\nabla F_j^{(l)}(\vx_j^{(l)})\right\|^2.
		\end{align*}
		Here, (a) applies as an instance of Jensen's inequality. Averaging the above inequality over $t=1,2,\ldots,T$, we obtain
		\begin{align*}
			\avetT \frac{|C_j^\tto|}{N} \|\vm_j^{(t)}\|^2
			&\leq \frac{1}{1-\beta} \frac{1}{N} \avetT\sum_{l=1}^t 
			|C_j^\tto| \beta^{t-l}\left\|\nabla F_j^{(l)}(\vx_j^{(l)})\right\|^2\\
			&\leq \frac{1}{T}\frac{1}{N}\frac{1}{1-\beta}\sum_{l=1}^T\sum_{t=l}^T |C_j^\tto| \beta^{t-l}\left\|\nabla F_j^{(l)}(\vx_j^{(l)})\right\|^2\\
			&\leq  \frac{1}{T}\frac{1}{N}\frac{1}{1-\beta}\sum_{l=1}^T\sum_{t=l}^T |C_j^{(l)}|\alpha^{t-l} \beta^{t-l}\left\|\nabla F_j^{(l)}(\vx_j^{(l)})\right\|^2\\
			&\leq \frac{1}{T}\frac{1}{1-\beta}\frac{1}{1-\alpha \beta} \sum_{l=1}^T \frac{|C_j^{(l)}|}{N} \left\|\nabla F_j^{(l)}(\vx_j^{(l)})\right\|^2
		\end{align*}
		Substituting the above inequality back into \eqref{eq:momentum-ineq-1}, we obtain
		\begin{multline*}
			\left(\frac{\gamma}{1-\beta}- \frac{L}{2} 
			\frac{\gamma^2}{(1-\beta)^2} -\frac{1}{2\epsilon}\frac{\gamma}{1-\beta} - \frac{\epsilon}{2}\frac{\alpha L^2\beta^2\gamma^3}{(1-\beta)^4(1-\alpha \beta)} \right) \avetT\sum_{j=1}^k \frac{|C_j^\tto|}{N} \|\nabla F_j^{(t)}(\vx_j^{(t)})\|^2
			\\
			\leq \frac{F(\vU^{(1)})-F(\vU^{(T+1)})}{T}.
		\end{multline*}
		We choose 
		$$
		\epsilon=\frac{(1-\beta)^\frac{3}{2}(1-\alpha\beta)^\frac{1}{2}}{\gamma\beta L \alpha^\frac{1}{2}}
		$$
		and rearrange the above inequality. Thus, we have
		\begin{multline*}
			\left(\frac{\gamma}{1-\beta}- \frac{L}{2} 
			\frac{\gamma^2}{(1-\beta)^2} -  \frac{\alpha^\frac{1}{2}L\beta\gamma^2}{2(1-\beta)^\frac{5}{2}(1-\alpha\beta)^\frac{1}{2}} \right) \avetT\sumjk \frac{|C_j^\tto|}{N}\|\nabla F_j^{(t)}(\vx_j^{(t)})\|^2
			\\ \leq\frac{F(\vU^{(1)})-F(\vU^{(T+1)})}{T}.
		\end{multline*}
		Since we initialize $\vm_j^{(0)}=\vzero$, we have
		\begin{align*}
			\vx_j^{(1)}&=\vx_j^{(0)}-\gamma \vm_j^{(0)} = \vx_j^{(0)},\\
			\vu_j^{(1)}&=\vx_j^{(0)}.
		\end{align*}
		Besides, since $F(\vU^{(T+1)})\geq F^*$, we have
		\begin{multline*}
			\left(\frac{\gamma}{1-\beta}- \frac{L}{2} 
			\frac{\gamma^2}{(1-\beta)^2} -  \frac{\alpha^\frac{1}{2}L\beta\gamma^2}{2(1-\beta)^\frac{5}{2}(1-\alpha\beta)^\frac{1}{2}} \right) \avetT\sumjk \frac{|C_j^\tto|}{N}\|\nabla F_j^{(t)}(\vx_j^{(t)})\|^2
			\\ \leq\frac{F(\vx_1^{(0)},\vx_2^{(0)},\ldots,\vx_k^{(0)})-F^*}{T}.
		\end{multline*}
		When
		$$
		\gamma\leq \min\left( \frac{1-\beta}{2L},\frac{(1-\beta)^\frac{3}{2}(1-\alpha\beta)^\frac{1}{2}}{2\alpha^\frac{1}{2}L\beta} \right),
		$$
		we have
		$$
		\avetT\sumjk \frac{|C_j^\tto|}{N}\|\nabla F_j^{(t)}(\vx_j^{(t)})\|^2
		\leq \frac{2(1-\beta)}{\gamma}\cdot
		\frac{F(\vx_1^{(0)},\vx_2^{(0)},\ldots,\vx_k^{(0)})-F^*}{T}.
		$$
	\end{proof}

	\section{Supplementary experiment details}
	\label{sec:supp-exp-detail}
	In this section, we provide details on the experiments described in Section \ref{sec:num-exp}.
	
	\subsection{Supplementary details for Section~\ref{sec:compare-product}}
	\label{sec:supp-detail-product}
	
	We elaborate on the generation of the synthetic data for the GPCA experiment in Section~\ref{sec:compare-product}. 
	\begin{itemize}
		\item First, we uniformly generate $k$ pairs of orthonormal vectors $\{\vepsilon_{1,j},\vepsilon_{2,j}\}$ for $ j=1,2,\ldots,k$. Each pair is generated uniformly at random, with $\vepsilon_{1,j}$ and $\vepsilon_{2,j}$ forming the basis of the $j$-th subspace.
		\item For each data point $i\in[N]$, we independently generate two Gaussian samples $\xi_{1,i}, \xi_{2,i}$. Next, we sample an index $j_i\in[k]$ uniformly at random. We then let $\vx_i=\xi_{1,i}\vepsilon_{1,j_i}+\xi_{2,i}\vepsilon_{2,j_i}$.
	\end{itemize}

	We provide in Algorithm~\ref{alg:Lloyd-GPCA} a detailed pseudo-code of Lloyd's algorithm for solving the GPCA problem in the sum-of-minimum formulation~\eqref{eq:GPCA}, which consists of two steps in each iteration, say updating the clusters via \eqref{eq:C-partition-update} and precisely compute the minimizer of each group objective function
	\begin{align*}
		\min_{\vvA_j^\top \vvA_j = I_r} \frac{1}{|C_j^{(t)}|} \sum_{i\in C_j^{(t)}} \|\vy_i^\top \vvA_j\|^2 & = \frac{1}{|C_j^{(t)}|} \sum_{i\in C_j^{(t)}} \textup{tr}\left(\vA_j^\top \bvy_i \bvy^\top \vA_j\right) \\
		& = \textup{tr}\left(\vA_j^\top \left(\frac{1}{|C_j^{(t)}|} \sum_{i\in C_j^{(t)}} \bvy_i \bvy^\top\right) \vA_j\right).
	\end{align*}
	
	\begin{algorithm}[htb!]
		\caption{Lloyd's Algorithm for generalized principal component analysis}\label{alg:Lloyd-GPCA}
		\begin{algorithmic}[1]
			\State Initialize $\vA_1^{(0)},\vA_2^{(0)},\dots,\vA_k^{(0)}$. Set $F^{(-1)}=+\infty$.
			\For{$t=0,1,2,\ldots,$ max iterations}
			\State Compute $F^\tto=F(\vA_1^\tto, \vA_2^\tto,\ldots,\vA_k^\tto)$.
			\If{$F^\tto = F^{(t-1)}$}
			\State Break.
			\EndIf
			\State Compute the partition $\{C_j^{(t)}\}_{j=1}^k$ via \eqref{eq:C-partition-update}.
			\For{$j=1,2,\ldots,k$}
			\If{$C_j^{(t)}\not= \emptyset$}
			\State Compute the matrix
			\begin{equation*}
				\frac{1}{|C_j^{(t)}|} \sum_{i\in C_j^{(t)}} \bvy_i \bvy_i^\top
			\end{equation*}
			and its $r$ orthonormal eigenvectors $\vv_1,\vv_2,\ldots,\vv_r$ corresponding to the smallest $r$ eigenvalues.
			\State Set 
			\begin{equation*}
				\vvA_j^{(t+1)} = \begin{bmatrix}
					\vv_1 & \vv_2 & \ldots & \vv_r
				\end{bmatrix}.
			\end{equation*}
			\Else
			\State $\vx_j^{(t+1)} = \vx_j^{(t)}$.
			\EndIf
			\EndFor
			\EndFor
		\end{algorithmic}
	\end{algorithm} 
	
	We implement the BCD algorithm \cite{peng2023block} for the following optimization problem:
	\begin{equation}\label{eq:GPCA-prod-appendix}
		\min_{\vvA_j^\top \vvA_j = I_r} \frac{1}{N}\sum_{i=1}^N \prod_{j\in[k]}\|\vy_i^\top \vvA_j\|^2.
	\end{equation}
	For any $j\in[k]$, when $\vvA_l$ is fixed for all $l\in[k]\backslash\{j\}$, the problem in \eqref{eq:GPCA-prod-appendix} is equivalent to:
	\begin{equation*}
		\min_{\vvA_j^\top \vvA_j = I_r} \frac{1}{N}\sum_{i=1}^N w_{ij}\|\vy_i^\top \vvA_j\|^2 =  \frac{1}{N}\sum_{i=1}^N w_{ij} \textup{tr} \left(\vvA_j^\top \vy_i\vy_i^\top \vvA_j\right) =  \textup{tr} \left(\vvA_j^\top \left(\frac{1}{N}\sum_{i=1}^N w_{ij} \vy_i\vy_i^\top\right) \vvA_j\right),
	\end{equation*}
	where the weights $w_{ij}$ are given by:
	\begin{equation*}
		w_{ij} = \prod_{l\neq j} \|\vy_i^\top \vvA_l\|^2.
	\end{equation*}
	The detailed pseudo-code can be found in Algorithm~\ref{alg:BCD-GPCA}.
	
	\begin{algorithm}[htb!]
		\caption{Block coordinate descent for generalized principal component analysis \cite{peng2023block}}\label{alg:BCD-GPCA}
		\begin{algorithmic}[1]
			\State Initialize $\vvA_1^{(0)},\vvA_2^{(0)},\dots,\vvA_k^{(0)}$. 
			\For{$t=0,1,2,\ldots,$ max iterations}
			\For{$j=1,2,\ldots,k$}
			\State Compute the weights:
			\begin{equation*}
				w_{ij}^{(t)} = \prod_{l< j} \|\vy_i^\top \vvA_l^{(t+1)}\|^2 \cdot \prod_{l>j} \|\vy_i^\top \vvA_l^{(t)}\|^2.
			\end{equation*}
			\State Compute the matrix:
			\begin{equation*}
				\frac{1}{N}\sum_{i=1}^N w_{ij}^{(t)} \vy_i\vy_i^\top
			\end{equation*}
			and its $r$ orthonormal eigenvectors $\vv_1,\vv_2,\ldots,\vv_r$ corresponding to the smallest $r$ eigenvalues.
			\State Set 
			\begin{equation*}
				\vvA_j^{(t+1)} = \begin{bmatrix}
					\vv_1 & \vv_2 & \ldots & \vv_r
				\end{bmatrix}.
			\end{equation*}
			\EndFor
			\EndFor
		\end{algorithmic}
	\end{algorithm} 
	
	\subsection{Supplementary details for Section~\ref{sec:compare-init}}
	\label{sec:supp-detail-init}
	\paragraph{Mixed linear regression}
	Here, we provide the detailed pseudo-code for Lloyd's algorithm used to solve $\ell_2$-regularized mixed linear regression problem in Section \ref{sec:num-exp}. Each iteration of the algorithm consists of two steps: reclassification and cluster parameter update. We alternatively reclassify indices $i$ to $C_j^\tto$ using \eqref{eq:C-partition-update} and update the cluster parameter $\vx_j^\tto$ for nonempty clusters $C_j^\tto$ using:
	\begin{equation}
		\label{eq:LMR-x-update}
		\vx_j^{(t+1)}=\left( \sum_{i\in C_j^\tto} \va_i\va_i^\top +\lambda |C_j^\tto|\vI \right)^{-1}\sum_{i\in C_j^\tto}b_i\va_i,
	\end{equation}
	so that $\vx_j^{(t+1)}$ is exactly the minimizer of the group objective function. The algorithm continues until $F^\tto$ stops decreasing after $\vx^\tto$'s update or a max iteration number is reached. The pseudo-code is shown in Algorithm \ref{alg:Lloyd-LMR}.
	
	\begin{algorithm}[htb!]
		\caption{Lloyd's Algorithm for mixed linear regression}\label{alg:Lloyd-LMR}
		\begin{algorithmic}[1]
			\State Initialize $\vx_1^{(0)},\vx_2^{(0)},\dots,\vx_k^{(0)}$. Set $F^{(-1)}=+\infty$.
			\For{$t=0,1,2,\ldots,$ max iterations}
			\State Compute $F^\tto=F(\vx_1^\tto, \vx_2^\tto,\ldots,\vx_k^\tto)$.
			\If{$F^\tto = F^{(t-1)}$}
			\State Break.
			\EndIf
			\State Compute the partition $\{C_j^{(t)}\}_{j=1}^k$ via \eqref{eq:C-partition-update}.
			\For{$j=1,2,\ldots,k$}
			\If{$C_j^{(t)}\not= \emptyset$}
			\State Compute $\vx_j^{(t+1)}$ using \eqref{eq:LMR-x-update}.
			\Else
			\State $\vx_j^{(t+1)} = \vx_j^{(t)}$.
			\EndIf
			\EndFor
			\EndFor
		\end{algorithmic}
	\end{algorithm}

	The dataset $\{(\va_i,b_i)\}_{i=1}^N$ for the $\ell_2$-regularized mixed linear regression is synthetically generated in the following way:
	\begin{itemize}
		\item Fix the dimension $d$ and the number of function clusters $k$, and sample $\vx_1^+,\vx_2^+,\dots,\vx_k^+\stackrel{\text{i.i.d.}}{\sim}\calN(0, \vI_d)$ as the linear coefficients of $k$ ground truth regression models.
		\item For $i=1,2,\dots,N$, we independently generate data $\va_i\sim\mathcal{N}(0,\vI_d)$, class index $c_i\sim\textup{Uniform}([k])$, noise $\epsilon_i\sim \mathcal{N}(0,\sigma^2)$, and compute $b_i = \va_i^\top \vx_{c_i}^+ +\epsilon_i$. 
	\end{itemize}
	In the experiment, the noise level is set to $\sigma=0.01$ and the regularization factor is set to $\lambda=0.01$.
	
	\paragraph{Mixed non-linear regression} The ground truth $\theta_j^+$'s are sampled from a standard Gaussian. The dataset $\{(\va_i,b_i)\}_{i=1}^N$ is generated in the same way as in the mixed linear regression experiment.
	We set the variance of the Gaussian noise on the dataset to $\sigma^2 = 0.01^2$ and use a regularization factor $\lambda=0.01$.
	
\end{document}